\documentclass{article}
\usepackage{amsmath,amsfonts,amssymb,amsthm,stmaryrd}
\usepackage{float}

\theoremstyle{plain}
\newtheorem{thrm}{Theorem}
\newtheorem{ass}[thrm]{Assumption}
\newtheorem{lmm}{Lemma}
\newtheorem{prpstn}{Proposition}

\theoremstyle{definition} 
\newtheorem{dfntn}{Definition}
\theoremstyle{remark}
\newtheorem{rmrk}{\textsc{Remark}}

\usepackage{tikz}
\usetikzlibrary{patterns}
\usepackage{pgfplots}
\pgfplotsset{compat=newest}
\usetikzlibrary{calc}

\newcommand{\logLogSlopeTriangle}[5]
{

    \pgfplotsextra
    {
        \pgfkeysgetvalue{/pgfplots/xmin}{\xmin}
        \pgfkeysgetvalue{/pgfplots/xmax}{\xmax}
        \pgfkeysgetvalue{/pgfplots/ymin}{\ymin}
        \pgfkeysgetvalue{/pgfplots/ymax}{\ymax}

        \pgfmathsetmacro{\xArel}{#1}
        \pgfmathsetmacro{\yArel}{#3}
        \pgfmathsetmacro{\xBrel}{#1-#2}
        \pgfmathsetmacro{\yBrel}{\yArel}
        \pgfmathsetmacro{\xCrel}{\xArel}

        \pgfmathsetmacro{\lnxB}{\xmin*(1-(#1-#2))+\xmax*(#1-#2)} 
        \pgfmathsetmacro{\lnxA}{\xmin*(1-#1)+\xmax*#1} 
        \pgfmathsetmacro{\lnyA}{\ymin*(1-#3)+\ymax*#3} 
        \pgfmathsetmacro{\lnyC}{\lnyA+#4*(\lnxA-\lnxB)}
        \pgfmathsetmacro{\yCrel}{\lnyC-\ymin)/(\ymax-\ymin)} 
        
        \coordinate (A) at (rel axis cs:\xArel,\yArel);
        \coordinate (B) at (rel axis cs:\xBrel,\yBrel);
        \coordinate (C) at (rel axis cs:\xCrel,\yCrel);

        \draw[#5]   (A)-- node[pos=0.5,anchor=north] {1}
                    (B)-- 
                    (C)-- node[pos=0.5,anchor=west] {#4}
                    cycle;
    }
}

\newcommand{\logLogSlopeTriangleinv}[5]
{

    \pgfplotsextra
    {
        \pgfkeysgetvalue{/pgfplots/xmin}{\xmin}
        \pgfkeysgetvalue{/pgfplots/xmax}{\xmax}
        \pgfkeysgetvalue{/pgfplots/ymin}{\ymin}
        \pgfkeysgetvalue{/pgfplots/ymax}{\ymax}

        \pgfmathsetmacro{\xArel}{#1}
        \pgfmathsetmacro{\yArel}{#3}
        \pgfmathsetmacro{\xBrel}{#1-#2}
        \pgfmathsetmacro{\yBrel}{\yArel}
        \pgfmathsetmacro{\xCrel}{\xBrel}

        \pgfmathsetmacro{\lnxB}{\xmin*(1-(#1-#2))+\xmax*(#1-#2)} 
        \pgfmathsetmacro{\lnxA}{\xmin*(1-#1)+\xmax*#1} 
        \pgfmathsetmacro{\lnyA}{\ymin*(1-#3)+\ymax*#3} 
        \pgfmathsetmacro{\lnyC}{\lnyA+#4*(\lnxA-\lnxB)}
        \pgfmathsetmacro{\yCrel}{(\lnyC-\ymin)/(\ymax-\ymin)} 
        
        \coordinate (A) at (rel axis cs:\xArel,\yArel);
        \coordinate (B) at (rel axis cs:\xBrel,\yBrel);
        \coordinate (C) at (rel axis cs:\xCrel,\yCrel);

        \draw[#5]   (A)-- node[pos=0.5,anchor=north] {1}
                    (B)-- node[pos=0.5,anchor=east] {#4}
                    (C)-- 
                    cycle;
    }
}

\title{Finite element method with local damage on the mesh
\footnote{This work has been carried out in the framework of Archim\`ede Labex (ANR-11-LABX-0033) and of
the A*MIDEX project (ANR-11-IDEX-0001-02), funded by the ``Investissements d’'Avenir'' French Government programme managed by the French National Research Agency (ANR). The authors acknowledge the support of the ANR project CroCo ANR-16-CE33-0008 and that of R\'egion Bourgogne Franche-Comt\'e ``Convention R\'egion 2015C-4991. Mod\`eles math\'ematiques et m\'ethodes num\'eriques pour l'\'elasticit\'e non-lin\'eaire''.}}
\author{Michel Duprez\footnote{Aix Marseille Universit\'{e}, CNRS, Centrale Marseille, I2M, UMR 7373, 13453 Marseille, France. e-mail: \texttt{mduprez@math.cnrs.fr}}
, ~Vanessa Lleras\footnote{IMAG, Univ Montpellier, CNRS, Montpellier, France.
\texttt{vanessa.lleras@umontpellier.fr}}
~ and ~Alexei Lozinski\footnote{Laboratoire de Math\'ematiques de Besan\c{c}on, UMR CNRS 6623,
Universit\'e Bourgogne Franche-Comt\'e,
16, route de Gray, 25030 Besan\c{c}on Cedex,
France. e-mail: \texttt{alexei.lozinski@univ-fcomte.fr}}}
\date{\today}
\begin{document}

\maketitle

\begin{abstract} 
We consider the finite element method on locally damaged meshes allowing for some distorted cells which are isolated from one another. 
In the case of the Poisson equation and piecewise  linear  Lagrange  finite  elements,
we show that the usual \textit{a priori} error estimates remain valid on such meshes. 
We also propose an alternative finite element scheme which is optimally convergent and, moreover, well conditioned, \textit{i.e.} its conditioning is of the same order as that of a standard finite element method on a regular mesh of comparable size.
\end{abstract}
%
%
%
%

\section{Introduction}

We are interested in the finite element method on meshes containing some isolated degenerated cells. 
The meshes of this type can be encountered in bio-mechanical applications where the objects with very complicated geometry (as a human face) should be meshed, and the mesh generators or mesh morphing techniques are not always able to satisfy the usual regularity constraints (see \textit{e.g.} \cite[p.3]{bucki2010fast}).
Our work is a preliminary study in which we propose a suitable finite element approximation in such situations without requiring to reconstruct a high quality mesh everywhere. We restrict ourselves to the simplest model: the Poisson equation with Dirichlet boundary conditions
\begin{equation}\label{eq:poisson}
\left\{\begin{array}{cl}
-\Delta u=f&\mbox{ in }\Omega,\\
 u=0&\mbox{ on }\partial\Omega
\end{array}\right.
\end{equation}
where $\Omega$ is a bounded polygonal (resp. polyhedral) domain in $\mathbb{R}^n$, $n=2$ (resp. $n=3$),  $\partial\Omega$ is its boundary,
and $f\in L^2(\Omega)$ is a given function. We only consider the standard piecewise linear continuous finite elements on a simplicial mesh without hanging nodes.
The formal (quite usual) definitions of the exact and approximated solutions to \eqref{eq:poisson} in the appropriate functional spaces  are given in the beginning of Section \ref{sec:a priori}.

The first goal of the present work is to highlight that we can recover the optimal convergence of the finite element method
even if the mesh contains several isolated almost degenerated simplexes.
More precisely, we shall assume that the majority of the simplexes in the mesh are regular in the usual Ciarlet sense \cite{ciarlet1978} but there are some distorted simplexes that are typically adjacent to regular mesh cells and well separated from one another by layers of regular cells.
The formal assumptions will be given in the beginning of Section \ref{sec:a priori}. To prove the optimal convergence of the standard finite element method,  we shall construct a modification of the nodal interpolation operator replacing the standard interpolating polynomial on a degenerated cell by another one obtained by averaging the interpolated function on a patch of cells surrounding the degenerated one.

Although the standard finite element method turns out optimally convergent  on the locally damaged meshes, as outlined above, it can suffer from bad conditioning of the stiffness matrix. Indeed, the gradient operator can have an arbitrary large norm on the space of piecewise polynomial functions on a mesh containing very elongated cells even if all the cells are of approximately the same diameter $h$.  The same issue of bad conditioning can be found in recent finite element methods on geometrically unfitted meshes such as CutFEM, cf. the review in \cite{Burman15}. The mesh is allowed to be cut by the domain boundary in this approach giving rise to eventually very narrow computational cells, and consequently to very ill-conditioned stiffness matrices. The workaround consists in introducing some stabilization terms which come in (at least) two forms: \textit{(i)} an augmented Lagrangian type approach involving the polynomial extension from ``good'' (uncut) to ``bad'' (cut) cells \cite{HaslingerRenard}; \textit{(ii)} \textit{ghost penalty} terms on the facets \cite{burmancras} of the cut cells which reduce the jumps betweens the gradients of the finite element solution between the ``bad'' and ``good'' cells. In both cases, the goal is to make the finite element solution on a ``bad'' cell to be aligned with its counterpart on a neighboring ``good'' cell. In the present paper, we inspire ourselves from both approaches described above and propose an alternative finite element discretization in which the approximated solution on the degenerated cells is made to be aligned with that on neighboring regular cells. We are able to prove that such a scheme is optimally convergent and well conditioned, \textit{i.e.} its conditioning is of the same order as that of a standard finite element method on a usual regular mesh of comparable size, provided the number of degenerated cells remains uniformly bounded.  \\


The present article is a contribution to the already rich literature studying the influence of the the mesh cell geometry on  the convergence of finite element approximations. The optimal $H^1$-convergence has been proved in \cite{zlamal1968} for second order elliptic equation and in \cite{zenisek} for linear elasticity equations under the \textit{minimum angle condition} in 2D:
there exists $\alpha_0\in(0,\pi)$ such that for any mesh cell $K$, 
\begin{equation}\label{eq:cond min angle}
0<\alpha_0\leqslant \alpha_K,
\end{equation}
where $\alpha_K$ is the minimum angle of  $K$. In \cite{brandts2, brandts3}, this condition was generalized to the higher dimensions.
If we denote by $h_K$ the diameter of $K$ and $\rho_K$ the diameter of the largest ball contained in $K$,
then \eqref{eq:cond min angle} is equivalent to already mentioned \textit{Ciarlet condition}  \cite{ciarlet1978}:
there exists $c_0$ such that for all mesh cells $K$
\begin{equation}\label{eq:cond ball}
h_{K}/\rho_{K}\leqslant c_0.
\end{equation}

The conditions above were further relaxed in several ways. 
Three groups (see \cite{BabuskaAziz1976,barnhill1976,jamet1976})
have proposed independently in 1976 a weaker assumption called the 
\textit{maximum angle condition}: there exists $\beta_0\in(0,\pi)$ such that 
for any mesh cell $K$ 
\begin{equation}\label{eq:cond max}
\beta_K\leqslant \beta_0<\pi,
\end{equation}
where $\beta_K$ is the maximum angle of $K$. The first condition \eqref{eq:cond min angle} implies the second \eqref{eq:cond max}. The second condition was generalized  for higher dimensions in \cite{jamet, kricek}. 

Furthermore, it is shown in \cite{annukainen2012} that even
the maximum angle condition may be not necessary. More precisely, if a degenerated  triangulation is included in a non-degenerated one, then optimal convergence rates.
The convergence on appropriate anisotropic meshes is studied in \cite{apel1999}. A sufficient condition for convergence (not necessarily of optimal order) was derived in \cite{kobayashi} under the name of the \textit{circumradius condition}: $\max_{K} R_K\to 0$ as $h\to 0$ where $R_K$ is the circumradius of the mesh cell $K$. Both the maximum angle and circumradius conditions for $O(h^\alpha)$ convergence are generalized in \cite{kucera}. It is proved that the triangulations can contain many elements violating these conditions as long as their maximum angle vertexes are sufficiently small size.  However, one cannot hope for an optimal convergence on completely arbitrary meshes: an example of a heavily distorted mesh family stemming from \cite{BabuskaAziz1976} has been recently analyzed in \cite{oswald2015} showing rigorously that the finite element method may fail to converge at all. The present paper propose yet another choice of assumptions on the mesh in the spirit of, but different from \cite{annukainen2012}, guaranteeing the optimal convergence.\\

The rest of the paper is organized as follows: 
in Section \ref{sec:a priori}, we prove that one can allow degenerated cells, which violate Condition 
\eqref{eq:cond min angle} or \eqref{eq:cond max}, if they are isolated in some sense.
We shall establish in Subsection \ref{sec:a priori normal} the optimal  $L^2$- and $H^1$-convergence of the standard finite elements method on such meshes. We also recall, in Subsection \ref{sec:poor}, the well known fact that the presence of degenerated cells may induce a large conditioning number of the stiffness matrix. We then propose in Section \ref{sec:alter} 
a modified finite element method that preserves the optimal convergence while ensuring a good conditioning.
We conclude with some numerical illustrations in Section \ref{sec:sim num}.

\section{Approximation  by linear finite elements under local mesh damage assumption}\label{sec:a priori}

Let us first recall the notions of the weak and approximated solutions to System \eqref{eq:poisson}. 
 We call a \textit{weak solution} in $V:=H^1_0(\Omega)$ to System \eqref{eq:poisson} a function $u\in V$ such that
\begin{equation}\label{cont prob}
a(u,v)=l(v)\mbox{ for all }v\in V,
\end{equation}
where  the bilinear form $a$ and the linear form $l$ are defined for all $u,v\in V$ by
$$a(u,v):=\int_{\Omega}\nabla u\cdot\nabla v~dx
\mbox{ \hspace*{1cm}and\hspace*{1cm} }l(v):=\int_{\Omega}fv~dx.$$
It is well known  that System \eqref{eq:poisson} admits a unique weak solution thanks to Lax-Milgram lemma.

Consider now a simplicial mesh  $\mathcal{T}_h$ on $\Omega$ without hanging nodes. This means that $\bar\Omega=\cup_{K\in\mathcal{T}_h}K$ with each mesh cell $K\in\mathcal{T}_h$ being a simplex (triangle in 2D, tetrahedron in 3D) and every two mesh cells $K_1,K_2\in\mathcal{T}_h$ being either disjoint or sharing a vertex, an edge, or a face (in 3D). We recall that $\rho_K$ denotes the diameter of the largest ball contained in a mesh cell $K$. Moreover, $h_\omega$ will denote the diameter of any bounded domain $\omega$ and we set $h=\max_{K\in\mathcal{T}_h}h_K$. As mentioned in the Introduction, we will assume that the cells of mesh $\mathcal{T}_h$ satisfy Ciarlet Condition \eqref{eq:cond ball} up to some isolated cells.

\begin{ass}\label{cond: plat}
Let $c_0>0$ and $K_1^{deg},...,K_I^{deg}$ be the \textit{degenerated cells} violating Ciarlet Condition \eqref{eq:cond ball}, \textit{i.e.} for $i\in\{1,...,I\}$
\begin{equation*}
K_i^{deg}\in\mathcal{T}_h
\mbox{ and }
h_{K_i^{deg}}/\rho_{K_i^{deg}}>c_0.
\end{equation*}
Each $K_i^{deg}$ is included in a patch $\mathcal{P}_i$, which is a union of mesh cells, star-shaped with respect to a ball of diameter $\rho_{\mathcal{P}_i}$ such that 
\begin{equation*}
h_{\mathcal{P}_i}/\rho_{\mathcal{P}_i}\leqslant c_1.
\end{equation*}
We denote by $\widetilde{\mathcal{P}}_i\supset{\mathcal{P}}_i$ the larger patch composed of mesh cells sharing at least a vertex with $\mathcal{P}_i$. Then
\begin{itemize}
\item The patches $\widetilde{\mathcal{P}}_i$ are mutually disjoint, \textit{i.e.} $\widetilde{\mathcal{P}}_i$ and $\widetilde{\mathcal{P}}_j$ have no common cells for $i\not= j$.
\item The number of cells in each $\widetilde{\mathcal{P}}_i$ is bounded by a constant $M$.
\end{itemize}
The intersection of boundaries $\partial\mathcal{P}_i$ and $\partial\Omega$ is either empty, or is reduced to a point, or is a subset of one side of the polygon/polyhedron $\Omega$ containing an $(n-1)$-dimensional ball of radius $\ge c_2h_{\mathcal{P}_i}$.
\end{ass}
\noindent\textbf{Notational warning.} In what follows, the letter $C$ will stand for constants which depend only on the generalized mesh regularity in the sense of Assumption \ref{cond: plat} (unless stated otherwise). This means that $C$ can depend on $c_0$, $c_1$, $c_2$, and $M$, but otherwise independent from the choice of mesh $\mathcal{T}_h$.  As usual, the value of $C$ can change from one line to another.

\medskip

An example of patches ${\mathcal{P}}_i$ and $\widetilde{\mathcal{P}}_i$ is given in Fig. \ref{fig:flatelement}. We illustrate there a typical situation of a degenerated triangle $K_i^{deg}$ (dashed in red) adjacent to a regular triangle $K_i^{nd}$ (dashed in grey). The patch  ${\mathcal{P}}_i$ is then formed of these two triangles $K_i^{deg}$ and $K_i^{nd}$. It is obviously star-shaped with respect to a ball (for example, the largest ball inscribed in $K_i^{nd}$). Its chunky parameter $h_{\mathcal{P}_i}/\rho_{\mathcal{P}_i}$ is close to that of surrounding regular triangles. Note, however, that Assumption \ref{cond: plat} allows for more general configurations, for example, a patch can contain several degenerated cells. 

\begin{figure}[H]
\begin{center}
\begin{tikzpicture}[scale=1.5]
\draw[-
] (1,1) -- (1,2);
\draw[-] (2.2,0) -- (1.7,1.3);
\draw[-] (2,2) -- (2,3);
\draw[-
] (3,1.8) -- (2,3);
\draw[-] (1,2) -- (2,3);
\draw[-
] (3,1) -- (3,1.8);
\draw[-
] (1.2,0) -- (2.2,0);
\draw[-] (1.7,1.3) -- (3,1);
\draw[-
] (1,2) -- (2,2);
\draw[-] (1,1) -- (2,2);
\draw[-] (1.2,0) -- (1,1);
\draw[-
] (1,1) -- (1.7,1.3);
\draw[-
] (1.7,1.3) -- (2,2);
\draw[-] (2,2) -- (3,1.8);
\draw[-
] (2.2,0) -- (3,1);
\draw[-] (1.2,0) -- (1.7,1.3);
\draw[-] (1.7,1.3) -- (3,1.8);
\draw[-
] (1.2,0) -- (0,0.5);
\draw[-] (0,0.5) -- (1,1);
\draw[-
] (0,0.5) -- (0,1.5);
\draw[-] (0,1.5) -- (1,1);
\draw[-] (0,1.5) -- (1,2);
\draw[-
] (0,1.5) -- (0.2,2.7);
\draw[-] (0.2,2.7) -- (1,2);
\draw[-
] (0.2,2.7) -- (2,3);
\fill [opacity=1,pattern=dots,pattern color=red] (1,1) -- (1.7,1.3) -- (2,2);
\fill [opacity=0.5,pattern=dots] (1,1) -- (1,2) -- (2,2);
\path (1.3,1.75) node {\small{$K^{nd}_i$}};
\path (2.3,1.7) node {\small{$K_{i,1}$}};
\path (2.7,1.4) node {\small{$K_{i,2}$}};
\path (2.3,0.8) node {\small{$K_{i,3}$}};
\path (1.65,0.4) node {\small{$K_{i,4}$}};
\path (1.25,0.8) node {\small{$K_{i,5}$}};
\path (0.8,0.5) node {\small{$K_{i,6}$}};
\path (0.5,1) node {\small{$K_{i,7}$}};
\path (0.6,1.5) node {\small{$K_{i,8}$}};
\path (0.5,2) node {\small{$K_{i,9}$}};
\path (1,2.5) node {\small{$K_{i,10}$}};
\path (1.7,2.2) node {\small{$K_{i,11}$}};
\path (2.3,2.2) node {\small{$K_{i,12}$}};
\end{tikzpicture}
\caption{Example of configuration: patch $\mathcal{P}_i$ (dashed), patch 
$\widetilde{\mathcal{P}}_i$ (all the cells), non-degenerated cell $K_i^{nd}$ (gray) of the patch $\mathcal{P}_i$.}
\label{fig:flatelement}
\end{center}\end{figure}
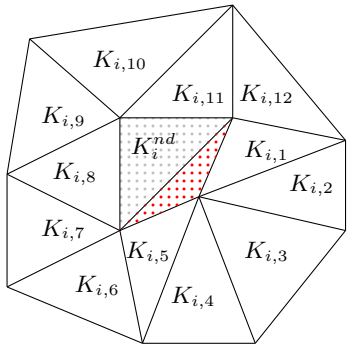

We now set the finite element space on mesh $\mathcal{T}_h$ and the finite element approximation to System \eqref{eq:poisson}. Let 
$$V_h:=\{v_h\in V:v_{h|K}\in \mathbb{P}_1(K)~\forall K\in \mathcal{T}_h\},$$
where $\mathbb{P}_1(K)$ is the space of polynomials of degree $\le 1$ on cell $K$.
Consider the following finite element approximation to System \eqref{cont prob}: find $u_h\in V_h$ such that:
\begin{equation}\label{discret prob}
a(u_h,v_h)=l(v_h)\mbox{ for all }v_h\in V_h.
\end{equation}

\subsection{\textit{A priori} error estimate}\label{sec:a priori normal}

In what follows, $|\cdot|_{i,A}$ and $\|\cdot\|_{i,A}$
denote the semi-norm and the norm associated to $H^i(A)$.

\begin{thrm}\label{th: a priori}
Let $u\in V$ and $u_h\in V_h$ be the solutions 
to System \eqref{cont prob} and System \eqref{discret prob}, respectively.
 Then, under Assumption \ref{cond: plat},
\begin{equation}\label{eq:a priori}
 | u - u_h |_{1,\Omega} \leq C 
  h | u |_{2,\Omega}.
\end{equation}
Moreover, if $\Omega$ is convex,
\begin{equation}\label{eq:a priori L2}
 \| u - u_h \|_{0,\Omega} \leq C
  h^2 | u |_{2,\Omega}.
\end{equation}
\end{thrm}

The proof of this theorem is completely standard (
\textit{cf.} \cite{ciarlet1978,brenner2007mathematical}) provided one has constructed an interpolant to $V_h$ satisfying the optimal error estimates. We thus 
go directly to the construction of such an interpolation operator which we shall call $\widetilde{\mathcal{I}}_h$ and properly introduce in Definition \ref{DefTildeIh}. The necessary properties of this operator will be established in the Proposition \ref{prop:global int}. We start with some technical lemmas.



\begin{lmm}\label{lemma:IntpP}
Under Assumption \ref{cond: plat}, for any $v\in H^2(\Omega)\cap H^1_0(\Omega)$ on any patch $\mathcal{P}_i$ there exists a polynomial $Q_h^i(v)$ on $\mathcal{P}_i$ of degree $\le 1$ vanishing on $\partial\mathcal{P}_i\cap\partial\Omega$ such that
\begin{equation}\label{eqlemma:IntpP}
 | v - Q_h^i(v) |_{1,\mathcal{P}_i}\leq C h_{\mathcal{P}_i} |v|_{2,\mathcal{P}_i}, 
 \quad
 \| v - Q_h^i(v) \|_{0,\mathcal{P}_i}\leq C h_{\mathcal{P}_i}^2 |v|_{2,\mathcal{P}_i}, 
 \quad
 \| v - Q_h^i(v) \|_{L^{\infty}(\mathcal{P}_i)}\leq C h_{\mathcal{P}_i}^{2-n/2} |v|_{2,\mathcal{P}_i}\,. 
\end{equation}
\end{lmm}
\begin{proof}
  We consider first the case of the patch $\mathcal{P}_i $ lying completely
  inside $\Omega$. We take then $Q_h^i (v)$ on $\mathcal{P}_i$ as the Taylor
  polynomial $Q^2 v$, \textit{cf.} Definition (4.1.3) from
  {\cite{brenner2007mathematical}}, averaged over the ball of diameter
  $\rho_{\mathcal{P}_i}$ mentioned in Assumption \ref{cond: plat}. The
  estimates (\ref{eqlemma:IntpP}) for $Q_h^i (v) = Q^2 v$ are thus given by
  Proposition (4.3.2) and Bramble-Hilbert Lemma (4.3.8) from
  {\cite{brenner2007mathematical}}.
  
  We now turn to the case when the boundary $\partial \mathcal{P}_i$
  intersects $\partial \Omega$ in only one point, say $x$. The polynomial
  $Q_h^i (v)$ should vanish at $x$ so that we correct \ $Q^2 v$ by
  subtracting from it its value at point. We set thus $Q_h^i (v) = Q^2 v -
  c_h$ where $c_h = Q^2 v (x)$. Since $v (x) = 0$, we have by the above
  mentioned properties of $Q^2 v$
  \[ | c_h | = | Q^2 v (x) - v (x) | \leqslant \| Q^2 v - v \|_{L^{\infty}
     (\mathcal{P}_i)} \leqslant Ch^{2 - n / 2} | v |_{2, \mathcal{P}_i} \]
  which entails
  \[ \| Q^i_h v - v \|_{0, \mathcal{P}_i} \leqslant \| Q^2 v - v \|_{0,
     \mathcal{P}_i} + | c_h | | \mathcal{P}_i |^{1 / 2} \leqslant Ch^2 | v
     |_{2, \mathcal{P}_i} \]
  and
  \[ \| Q^i_h v - v \|_{L^{\infty} (\mathcal{P}_i)} \leqslant \| Q^2 v - v
     \|_{L^{\infty} (\mathcal{P}_i)} + | c_h | \leqslant Ch^{} | v |_{2,
     \mathcal{P}_i} \, . \]
  The $H^1$ semi-norm of the error is not affected by the constant $c_h$, so
  that the announced estimate for $| Q^i_h v - v |_{1, \mathcal{P}_i}$ is also
  valid.
  
  The last case to consider is when $\partial \mathcal{P}_i$ has a non-empty
  intersection with a side, say $\Gamma$, of $\partial \Omega$, which is not
  reduced to one point. We recall that \ $\partial \mathcal{P}_i \cap \Gamma$
  is assumed then to contain a ball of radius of order $h_{\mathcal{P}_i}$. We
  introduce the polynomial $c_h$ of degree $\leqslant 1$ that coincides with
  $Q^2 v$ on $\partial \mathcal{P}_i \cap \Gamma$ and does not vary in the
  direction perpendicular to $\Gamma$. Setting $Q_h^i (v) = Q^2 v - c_h$ we
  see immediately that $Q_h^i (v)$ vanishes on $\Gamma$. Moreover, thanks to
  our geometrical assumptions and the fact that $v$ vanishes on $\partial
  \mathcal{P}_i \cap \Gamma$,
  \[ \| c_h \|_{L^{\infty} (\mathcal{P}_i)} \leqslant C \| c_h \|_{L^{\infty}
     (\partial \mathcal{P}_i \cap \Gamma)} = C \| Q^2 v - v \|_{L^{\infty}
     (\partial \mathcal{P}_i \cap \Gamma)} \leqslant Ch_{\mathcal{P}_i}^{2 - n
     / 2} | v |_{2, \mathcal{P}_i} \, . \]
  We can thus prove the desired estimates for $\| Q^i_h v - v \|_{0,
  \mathcal{P}_i}$ and $\| Q^i_h v - v \|_{L^{\infty} (\mathcal{P}_i)}$ as in
  the previous case. Finally, by an inverse inequality,
  \[ \| \nabla c_h \|_{L^{\infty} (\mathcal{P}_i)} \leqslant
     \frac{C}{h_{\mathcal{P}_i}} \| c_h \|_{L^{\infty} (\partial \mathcal{P}_i
     \cap \Gamma)} \leqslant Ch_{\mathcal{P}_i}^{1 - n / 2} | v |_{2,
     \mathcal{P}_i} \]
  so that
  \[ | Q^i_h v - v |_{1, \mathcal{P}_i} \leqslant | Q^2 v - v |_{1,
     \mathcal{P}_i} + \| \nabla c_h \|_{L^{\infty} (\mathcal{P}_i)} |
     \mathcal{P}_i |^{1 / 2} \leqslant Ch | v |_{2, \mathcal{P}_i} \, . \]
\end{proof}

We also recall the usual interpolation error estimates on regular cells for the standard Lagrange interpolation operator ${\mathcal{I}}_h$ to the space of piecewise linear functions, \textit{cf.} \cite{ciarlet1978,brenner2007mathematical}.
\begin{lmm}\label{lemma:IntpK}
Under Assumption \ref{cond: plat}, we have on each mesh cell $K\in\mathcal{T}_h$ outside of patches $\mathcal{P}_i$
\begin{equation}
  | v - {\mathcal{I}}_h(v) |_{1,K}\leq C h |v|_{2,K}, 
 \quad
  \| v - {\mathcal{I}}_h(v) \|_{0,K} \le Ch^2 |v|_{2,K}, 
 \quad
 \| v - {\mathcal{I}}_h(v) \|_{L^{\infty}(K)}\leq C h^{2-n/2} |v|_{2,K}, 
\end{equation}
for any $v\in H^2(K)$.
\end{lmm}

\begin{dfntn}\label{DefTildeIh}
For all $v \in H^2(\Omega)\cap H^1_0(\Omega)$,  let $\widetilde{\mathcal{I}}_h (v)$ be the function in $V_h$ that coincides with $Q_h^i(v)$ from Lemma \ref{lemma:IntpP} on each patch $\mathcal{P}_i$, and with the standard Lagrange interpolation ${\mathcal{I}}_hv$ on all the cells $K\in\mathcal{T}_h$ out of the extended patches $\widetilde{\mathcal{P}}_i$, \textit{i.e.} $\widetilde{\mathcal{I}}_h (v)(x)=v(x)$ at all the mesh nodes $x\in\bar\Omega\setminus\cup_{i\in\{1,\ldots,I\}}\widetilde{\mathcal{P}}_i$. 
\end{dfntn}

Note that $\widetilde{\mathcal{I}}_h (v)$ is uniquely defined also on the mesh cells from $\widetilde{\mathcal{P}}_i\setminus{\mathcal{P}}_i$, $i=1,\ldots,I$ although they are not explicitly mentioned above. Indeed, all the vertices of such cells are shared either with a patch ${\mathcal{P}}_i$ or with a regular cell from $\bar\Omega\setminus\cup_{i\in\{1,\ldots,I\}}\widetilde{\mathcal{P}}_i$.   Since the values of $\widetilde{\mathcal{I}}_h (v)$ are given at all these nodes by the definition above, the piecewise linear function   $\widetilde{\mathcal{I}}_h (v)$ is well defined everywhere.



We now prove the global interpolation estimates for the interpolation operator $\widetilde{\mathcal{I}}_h$.

\begin{prpstn}\label{prop:global int}
Under Assumption \ref{cond: plat}, we have for all $v\in H^2(\Omega)\cup H^1_0(\Omega)$
$$
| v - \widetilde{\mathcal{I}}_h(v) |_{1,\Omega}\leq C h |v|_{2,\Omega} ,
\quad
 \| v - \widetilde{\mathcal{I}}_h(v) \|_{0,\Omega}\leq C h^2 |v|_{2,\Omega}.
$$
\end{prpstn}
\begin{proof}
The contributions to the interpolation errors on the patches $\mathcal{P}_i$ and on the mesh cells outside the patches $\widetilde{\mathcal{P}}_i$ (where the interpolators $\widetilde{\mathcal{I}}_h$ and ${\mathcal{I}}_h$ coincide) are already covered by Lemmas \ref{lemma:IntpP} and \ref{lemma:IntpK}. It remains to bound the error on mesh cells in $\widetilde{\mathcal{P}}_i\setminus{\mathcal{P}}_i$.

Let $K\in\mathcal{T}_h$ and $K\subset\widetilde{\mathcal{P}}_i\setminus{\mathcal{P}}_i$.  By the triangle inequality and Lemma \ref{lemma:IntpK}, 
\begin{equation*}
 | v - \widetilde{\mathcal{I}}_h(v) |_{1,K}
 \leq  | v - \mathcal{I}_h(v) |_{1,K} + | \mathcal{I}_h(v) - \widetilde{\mathcal{I}}_h(v) |_{1,K}
 \leq  Ch| v |_{2,K} + | r_h |_{1,K},
\end{equation*}
where we have denoted $r_h:=\mathcal{I}_h(v) - \widetilde{\mathcal{I}}_h(v)$. By a homogeneity argument and the equivalence of norms on finite dimensional space, we see easily
$$
| r_h |_{1,K} \le Ch^{n/2-1}\| r_h \|_{L^\infty(K)}.
$$
Recalling that $r_h$ is a polynomial of degree $\le 1$ vanishing at the vertices of $K$ on $\partial\widetilde{\mathcal{P}}_i$, the other vertices belonging to $\partial{\mathcal{P}}_i$, we conclude 
$$
\| r_h \|_{L^\infty(K)} \le
\| r_h \|_{L^\infty(\partial K \cap \partial{\mathcal{P}}_i) } \le
\|v - \mathcal{I}_h(v)\|_{L^\infty(K)} + \|v - \widetilde{\mathcal{I}}_h(v)\|_{L^\infty(\mathcal{P}_i)}
\le
C h^{2-n/2} (|v|_{2,K} + |v|_{2,{\mathcal{P}}_i}).
$$
Putting the estimates above together yields
\begin{equation}\label{estK H1}
 | v - \widetilde{\mathcal{I}}_h(v) |_{1,K}
 \leq  C h (|v|_{2,K} + |v|_{2,{\mathcal{P}}_i}).
\end{equation}
Similarly,
\begin{equation}\label{estK L2}
 \| v - \widetilde{\mathcal{I}}_h(v) \|_{0,K}
 \leq  C h^2 (|v|_{2,K} + |v|_{2,{\mathcal{P}}_i}).
\end{equation}

Taking the square on both sides of \eqref{estK H1} and \eqref{estK L2}, summing them over all the mesh cells $K\subset\widetilde{\mathcal{P}}_i\setminus{\mathcal{P}}_i$, $i=1,\ldots,I$ (recall that the number of such cells on each patch is bounded by a predefined constant $M$), adding the estimates from lemma \ref{lemma:IntpP} on the patches $\mathcal{P}_i$ and those of Lemma \ref{lemma:IntpK} on the mesh cells outside the patches $\widetilde{\mathcal{P}}_i$ gives the desired result. 
\end{proof}


\subsection{Poor conditioning of the system matrix}\label{sec:poor}

In this section, we shall recall the well known fact that the presence of degenerated cells can induce an arbitrary large conditioning number of the associated finite element matrix. In the following proposition, we consider a particular example of a mesh satisfying Assumption \ref{cond: plat} and give an estimator for the conditioning number. This result should be contrasted with the ``normal'' conditioning number of order $1/h^2$ on a quasi-uniform mesh. 

\begin{prpstn}\label{prop:poor}
Suppose that the mesh $\mathcal{T}_h$ satisfies Assumption \ref{cond: plat} and contains a degenerate cell $K^{deg}$ such that 
\begin{equation}\label{eq:cond poor}
\rho_{K^{deg}}= \varepsilon,\quad
h_{K^{deg}}\geqslant C_1h.
\end{equation}
Then the conditioning number
$\boldsymbol{\kappa}(\boldsymbol{A}):=\|\boldsymbol{A}\|_2\|\boldsymbol{A}^{-1}\|_2$ 
of the matrix $\boldsymbol{A}$ associated to 
the bilinear form $a$ in $V_h$ satisfies
\begin{equation*}
\boldsymbol{\kappa}(\boldsymbol{A})
\geqslant \frac{C}{h\varepsilon}
\end{equation*}
for sufficiently small $h$, with $C$ depending only on $C_1$ and $\Omega$. Here, $\|\cdot\|_2$ stands for the matrix norm associated to the vector 2-norm.
\end{prpstn}

\begin{proof}
Denote by $N$ the dimension of $V_h$.
Consider $\phi_h$ the basis function of $V_h$  equal to $1$ 
at the node of $K^{deg}$ opposite to the largest edge (face) of $K^{deg}$, vanishing at all the other nodes,
and $\boldsymbol{\phi}\in\mathbb{R}^N$ the vector representing $\phi_h$ in the basis of hat functions.
Then, denoting by $|\cdot|_2$ the vector 2-norm on $\mathbb{R}^N$ and by $(\cdot,\cdot)$ the associated inner product,
$$\|\boldsymbol{A}\|_2
=
\sup_{\boldsymbol{u}\in\mathbb{R}^N}\frac{(\boldsymbol{A}\boldsymbol{u},\boldsymbol{u})}{|\boldsymbol{u}|_2^2}
\geqslant (\boldsymbol{A}\boldsymbol{\phi},\boldsymbol{\phi})
=a(\phi_h,\phi_h)=|\phi_h|_{1,\Omega}^2\geqslant |\phi_h|_{1,K^{deg}}^2.
$$
By \eqref{eq:cond poor}, the gradient of $\phi_h$ is of order $1/\varepsilon$ on $K^{deg}$, and the area of $K^{deg}$ is of order $\varepsilon h^{n-1}$. Thus,
\begin{equation*}
\|\boldsymbol{A}\|_2 \geqslant
 |\phi_h|_{1,K^{deg}}^2\geqslant C\dfrac{h^{n-1}}{\varepsilon}.
\end{equation*}

Now take any $\psi\in H^2(\Omega)\cap H^1_0(\Omega)$, $\psi\not =0$
and let $\boldsymbol{\psi}\in\mathbb{R}^N$ be the vector associated to $\widetilde{\mathcal{I}}_h\psi$. Then
$$\|\boldsymbol{A}^{-1}\|_2
=\sup_{\boldsymbol{u}\in\mathbb{R}^N}
\frac{|\boldsymbol{u}|_2^2}{(\boldsymbol{A}\boldsymbol{u},\boldsymbol{u})}
\geqslant
\frac{|\boldsymbol{\psi}|_2^2}{(\boldsymbol{A}\boldsymbol{\psi},\boldsymbol{\psi})}
=
\frac{|\boldsymbol{\psi}|_2^2}{a(\widetilde{\mathcal{I}}_h\psi,\widetilde{\mathcal{I}}_h\psi)}
\geqslant
\frac{C}{h^n}\frac{\|\widetilde{\mathcal{I}}_h\psi\|_{0,\Omega}^2}{|\widetilde{\mathcal{I}}_h\psi|_{1,\Omega}^2}.
$$
We have used here the bound
$$\|v_h\|_{0,\Omega}^2\leqslant Ch^{n}|\boldsymbol{v}|^2
$$
valid for any $v_h\in V_h$ and the corresponding vector $\boldsymbol{v}$ since all the mesh cells are of diameter $\le h$.
Proposition \ref{prop:global int} implies
\begin{equation*}
\left\{\begin{array}{l}
\|\widetilde{\mathcal{I}}_h\psi\|_{0,\Omega}\geqslant \|\psi\|_{0,\Omega}-Ch^2|\psi|_{2,\Omega}\geqslant C\|\psi\|_{0,\Omega},\\
|\widetilde{\mathcal{I}}_h\psi|_{1,\Omega}\leqslant |\psi|_{1,\Omega}+Ch|\psi|_{2,\Omega}\leqslant C|\psi|_{1,\Omega}.
\end{array}
\right.
\end{equation*}
So that
$$\|\boldsymbol{A}^{-1}\|_2
\geqslant
\frac{C}{h^n}.
$$
This gives the desired result.
\end{proof}

\section{A well conditionned alternative finite element scheme}\label{sec:alter}

In this section, we build an alternative finite element method for which  the optimal convergence rates \eqref{eq:a priori} and \eqref{eq:a priori L2} hold true and the conditioning number of the finite element matrix is of order $C/h^2$ if all the mesh cells are of diameter $\sim h$. We start by the observation that such a method could be based on a subspace $\widetilde{V}_h\subset V_h$ which is the image of interpolation operator $\widetilde{\mathcal{I}}_h$, \textit{i.e.} 
$$\widetilde{V}_h:=\{v_h\in V_h:[\nabla v_h]_{|F}=0
\mbox{ for all }F\in\mathcal{F}_i,i\in\{1,...,I\}\} ,$$
where  $\mathcal{F}_i$ is the set of interior edges (faces) 
of the patch $\mathcal{P}_i$ and $[\cdot]_{|F}$ represents the jump on $F$. In view of our interpolation estimates, the problem of  finding $\widetilde{u}_h\in \widetilde{V}_h$ such that
\begin{equation*}
a(\widetilde{u}_h,\widetilde{v}_h)=l(\widetilde{v}_h)
\mbox{ for all }\widetilde{v}_h\in \widetilde{V}_h
\end{equation*}
would produce an approximate solution with optimal error. Moreover, it is easy to see that the matrix would be well-conditioned since the space $\widetilde{V}_h$ ignores the degenerated cells. Such a method is only of theoretical interest because one cannot easily construct a basis for $\widetilde{V}_h$ using available finite element libraries. In what follows, we use this problem rather as an inspiration in constructing an implementable finite element scheme.

In doing so, we shall impose further restrictions on the mesh:
\begin{ass}\label{cond: plat2} The mesh satisfies Assumption \ref{cond: plat}. Moreover, 
\begin{itemize}
\item The number of patches $I$ is bounded by some $I_{\max}$.
\item Each patch $\mathcal{P}_i$ contains a non-degenerated cell $K_i^{nd}$, \textit{i.e.}  such that  $h_{\mathcal{P}_i}/\rho_{K_i^{nd}}\leqslant c_0$.
\end{itemize}
\end{ass}
In what follows, the constants $C$ will be allowed to depend on the additional parameter in Assumption \ref{cond: plat2}, \textit{i.e.} $I_{\max}$.


We shall need the following modification of the previously defined interpolation operator $\widetilde{\mathcal{I}}_h$, which makes sense under Assumption \ref{cond: plat2}, \textit{cf.} also Fig. \ref{fig:flatelement}, and will be incorporated explicitly into our modified finite element scheme. 
\begin{dfntn}\label{DefHatIh}
For all $v \in H^2(\Omega)\cap H^1_0(\Omega)$,  let $\widehat {\mathcal{I}}_h (v)$ be the function in $V_h$ that coincides with with the standard Lagrange interpolation ${\mathcal{I}}_hv$ on all the cells $K\in\mathcal{T}_h$ out of the extended patches $\widetilde{\mathcal{P}}_i$, 
and is given on each patch ${\mathcal{P}}_i$, not touching the boundary $\partial\Omega$, by
\begin{equation}\label{IntpExt}
  \widehat{\mathcal{I}}_h (v) |_{\mathcal{P}_i} 
  := \mbox{Ext}\, \left ( \mathcal{I}_h (v) |_{K^{nd}_i} \right ),
\end{equation}
where $\mathcal{I}_h$ stands again for the standard Lagrange interpolation operator on $K_i^{nd}$, and Ext stands for the extension of a polynomial from $K^{nd}_i \subset \mathcal{P}_i$ to the whole $\mathcal{P}_i$ without changing the coefficients of the polynomial. If the patch $\mathcal{P}_i$ touches $\partial\Omega$, then $\widehat {\mathcal{I}}_h (v)$ is also based there on formula (\ref{IntpExt}), corrected as in Lemma \ref{lemma:IntpP}.
  
\end{dfntn}

\begin{rmrk}\label{rem:global int}
The new interpolation operator $\widehat {\mathcal{I}}_h (v)$ satisfies the same optimal estimates as that for the old operator $\widetilde {\mathcal{I}}_h (v)$ which are given in Proposition \ref{prop:global int}, the proof of which is based on Lemma \ref{lemma:IntpP}. 
To prove that Lemma \ref{lemma:IntpP} remains valid for $\widehat{\mathcal{I}}_h$, i.e. redefining in (\ref{eqlemma:IntpP}) the original $Q_h^i$ bye $Q_h^i:=\widehat{\mathcal{I}}_h (v) |_{\mathcal{P}_i}$ as in (\ref{IntpExt}), we refer to Theorem (4.4.4) and Corollary (4.4.7) from \cite{brenner2007mathematical}. Following their proofs, one can see  that the only thing to check is the boundedness of operator Ext in \eqref{IntpExt} as a linear map on the space of polynomials of degree $\le 1$ equipped with the norm of $L^\infty(K_i^{nd})$ to $L^\infty({\mathcal{P}_i})$. This, in turn, follows easily from our geometrical Assumptions \ref{cond: plat}, \ref{cond: plat2}.
\end{rmrk}

\subsection{An alternative scheme}

We denote by $a_{\omega}$ the restriction of $a$ on a subset $\omega$ of $\Omega$, and by $(\cdot,\cdot)_{\omega}$ the inner product in  $L^2(\omega)$.
Consider the bilinear form $a_h$ defined for all $u_h,~v_h\in V_h$ by
\begin{equation}\label{eq:bil GP}
 a_h (u_h, v_h) := a_{\Omega^{nd}_h} (u_h, v_h) 
 + \sum_i a_{\mathcal{P}_i}(\widehat{\mathcal{I}}_h u_h, \widehat{\mathcal{I}}_h v_h)
  + \sum_i  \frac{1}{h_{\mathcal{P}_i}^2}((\mbox{Id}- \widehat{\mathcal{I}}_h) u_h, (\mbox{Id}-
   \widehat{\mathcal{I}}_h) v_h)_{\mathcal{P}_i},
   \end{equation}
where $\Omega^{nd}_h:=\Omega\backslash(\cup_i\overline{\mathcal{P}_i})$ and the interpolation operator $\widehat{\mathcal{I}}_h$ is defined by \eqref{IntpExt}, \textit{i.e.} $u_h$ is not used directly inside the patches in the second term of $a_h$, but rather it is extended from a non-degenerate cell inside each patch. The third term in $a_h$ will serve, loosely speaking, to penalize the eventual non-alignment of the approximate solution $u_h$ with the optimal subspace $\widetilde{V}_h$. 

We now introduce the following method approximating System \eqref{cont prob}: find $u_h\in V_h$ such that
\begin{equation}\label{disc prob GP}
a_h(u_h,v_h)=l(v_h)\mbox{ for all }v_h\in V_h.
\end{equation}
The idea of using the polynomial extension from ``good'' to ``bad'' mesh cells in the scheme (\ref{disc prob GP}) is borrowed from \cite{HaslingerRenard}. We shall also see that the scheme can be recast in a form using the interior penalization on the mesh facets between ``good'' and ``bad'' cells, as in the ghost penalty method \cite{burmancras}. 


\subsection{\textit{A priori} estimate}



The approximation of System \eqref{cont prob} by \eqref{disc prob GP}
induces a quasi-optimal convergence rate:

\begin{thrm}[\textit{A priori} estimate]\label{theo:a priori gp}
Let $u\in V$ and $u_h\in V_h$ be the solutions 
to System \eqref{cont prob} and System \eqref{disc prob GP}, respectively.
 Then, under Assumption \ref{cond: plat2}, we have for any $\varepsilon>0$ if $n=2$, and for $\varepsilon=0$ if $n=3$, 
\begin{equation}\label{eq:a priori GP}
 |u-\Pi_hu_h|_{1,\Omega}:= |u-u_h|_{1,\Omega^{nd}_h} +\sum_i|u-\widehat{\mathcal{I}}_hu_h|_{1,\mathcal{P}_i} \leq C 
  h^{1-\varepsilon} | u |_{2,\Omega},
\end{equation}
where $\Pi_hu_h$ is equal to $u_h$ on $\Omega^{nd}_h$ and $\widehat{\mathcal{I}}_hu_h$ on $\mathcal{P}_i$.
Moreover, if $\Omega$ is convex,
\begin{equation*}
 \| u - u_h \|_{0,\Omega} \leq C
  h^{2-\varepsilon} | u |_{2,\Omega}.
\end{equation*}
In the case $n=2$, $C$ depends on $\varepsilon$ (in addition to its dependence on the mesh regularity).
\end{thrm}

Before proving Theorem \ref{theo:a priori gp}, we first give some auxiliary results.
\begin{lmm}[Galerkin orthogonality]\label{lemma:galer}
Consider $u$ and $u_h$ the solution to Systems \eqref{cont prob} and \eqref{disc prob GP}. Then
$$a_{\Omega^{nd}_h} (u_h-u,v_h)
- \sum_i a_{\mathcal{P}_i} (u, v_h) 
+\sum_i a_{\mathcal{P}_i}(\widehat{\mathcal{I}}_h u_h, \widehat{\mathcal{I}}_hv_h)
+\sum_i\frac{1}{h_{\mathcal{P}_i}^2} ((\mbox{Id}-\widehat{\mathcal{I}}_h) u_h,(\mbox{Id}- \widehat{\mathcal{I}}_h)v_h)_{0,\mathcal{P}_i}=0,$$
for all $v_h\in V_h$.
\end{lmm}

The proof of Lemma \ref{lemma:galer} is immediate.

We shall need the norm $\interleave\cdot\interleave$ defined for all $v_h\in V_h$ by 
$$\interleave v_h\interleave:=a_h(v_h,v_h)^{1/2}
=\left(
|v_h|^2_{1,\Omega^{nd}_h} 
 + \sum_i |\widehat{\mathcal{I}}_h v_h|^2_{1,\mathcal{P}_i}
  + \sum_i\frac{1}{h_{\mathcal{P}_i}^2} \|v_h- \widehat{\mathcal{I}}_h v_h \|^2_{0,\mathcal{P}_i}
\right)^\frac{1}{2}.$$
Note for the future use that this norm is also well defined on $V\cap H^2(\Omega)$.

\begin{lmm}\label{lemma:patch}
Under Assumption \ref{cond: plat2},
for all $v_h\in V_h$, it holds
\begin{equation*}
\sum_i|v_h-\widehat{\mathcal{I}}_hv_h|_{1,\widetilde{\mathcal{P}}_i\backslash\mathcal{P}_i}^2
\leqslant C \interleave v_h \interleave^2.
\end{equation*}
\end{lmm}

\begin{proof}  
It suffices to prove for each patch
\begin{equation}\label{eq:lemma patch}
|v_h-\widehat{\mathcal{I}}_hv_h|_{1,\widetilde{\mathcal{P}}_i\backslash\mathcal{P}_i}^2
\leqslant C\left(
|v_h|_{1,\widetilde{\mathcal{P}}_i\backslash\mathcal{P}_i}^2
+|\widehat{\mathcal{I}}_hv_h|_{1,\mathcal{P}_i}^2
+\frac{1}{h_{\mathcal{P}_i}^2}\|(\mbox{Id}-\widehat{\mathcal{I}}_h)v_h\|_{0,\mathcal{P}_i}^2
\right).
\end{equation}
Using the fact that the number of cells in each $\widetilde{\mathcal{P}}_i$ is bounded by a constant $M$
(see Assumption \ref{cond: plat}),
 the functional space involved in \eqref{eq:lemma patch} is of finite dimension and the geometry of the patch $\widetilde{\mathcal{P}}_i\backslash\mathcal{P}_i$ is governed by a finite number of parameters, the existence of the constant $C$ will follow from maximization of the ratio of the left-hand side to the right-hand side over all polynomials and all the acceptable geometries and from a homogeneity argument.
We need only to verify that if the right-hand side  vanishes, \textit{i.e.}
$$|v_h|_{1,\widetilde{\mathcal{P}}_i\backslash\mathcal{P}_i}^2
+|\widehat{\mathcal{I}}_hv_h|_{1,\mathcal{P}_i}^2
+\frac{1}{h_{\mathcal{P}_i}^2}\|(\mbox{Id}-\widehat{\mathcal{I}}_h)v_h\|_{0,\mathcal{P}_i}^2=0$$
on a patch $\mathcal{P}_i$ and for some function $v_h\in V_h$, then  the left-hand side vanishes as well, \textit{i.e.} $v_h-\widehat{\mathcal{I}}_hv_h$ {is constant} on $\widetilde{\mathcal{P}}_i\backslash\mathcal{P}_i$.
Since $|v_h|_{1,\widetilde{\mathcal{P}}_i\backslash\mathcal{P}_i}=0$ and $|\widehat{\mathcal{I}}_hv_h|_{1,\mathcal{P}_i}=0$
$$
v_h=C_i\mbox{ in }\widetilde{\mathcal{P}}_i\backslash\mathcal{P}_i,
\quad
\widehat{\mathcal{I}}_hv_h=D_i\mbox{ in }\mathcal{P}_i
$$
with some constants $C_i$ and $D_i$.
Since $\|(\mbox{Id}-\widehat{\mathcal{I}}_h)v_h\|_{0,\mathcal{P}_i}=0$, we deduce that 
$$v_h=\widehat{\mathcal{I}}_hv_h=C_i=D_i\mbox{ in }\mathcal{P}_i$$
so that $(\mbox{Id}-\widehat{\mathcal{I}}_h)v_h=0$ on $\partial\mathcal{P}_i$. This entails 
$(\mbox{Id}-\widehat{\mathcal{I}}_h)v_h=0$ in $\widetilde{\mathcal{P}}_i\backslash\mathcal{P}_i$ by construction of $\widehat{\mathcal{I}}_h$.
\end{proof}

\begin{lmm}\label{NewLemma} 
Under Assumption \ref{cond: plat2},
 {we have for any $\varepsilon>0$ if $n=2$, and for $\varepsilon=0$ if $n=3$,}
\begin{equation*}
\left( \sum_i |u|_{1,\widetilde{\mathcal{P}}_i \setminus \mathcal{P}_i}^2 \right)^{\frac{1}{2}}  \leqslant C_{\varepsilon} h^{1 - \varepsilon}\| u \|_{2, \Omega} 
\end{equation*}
for all $u\in H^2(\Omega)$.
\end{lmm}
\begin{proof}
  Let $\overline{\nabla u}^i$ denote the average of $\nabla u$ on
  $\widetilde{\mathcal{P}}_i \setminus \mathcal{P}_i$, \textit{i.e.}
  \[ \overline{\nabla u}^i := \frac{1}{| \widetilde{\mathcal{P}}_i \setminus
     \mathcal{P}_i |} \int_{\widetilde{\mathcal{P}}_i \setminus \mathcal{P}_i}
     \nabla u. \]
  Then, by Poincar\'e inequality, it holds
  \[ |u|_{1, \widetilde{\mathcal{P}}_i \setminus \mathcal{P}_i} \leqslant \|
     \nabla u - \overline{\nabla u}^i \|_{0, \widetilde{\mathcal{P}}_i
     \setminus \mathcal{P}_i} + \| \overline{\nabla u}^i \|_{0,
     \widetilde{\mathcal{P}}_i \setminus \mathcal{P}_i} \leqslant Ch | u |_{2,
     \widetilde{\mathcal{P}}_i \setminus \mathcal{P}_i} + \| \overline{\nabla
     u}^i \|_{0, \widetilde{\mathcal{P}}_i \setminus \mathcal{P}_i}. \]

We now consider separately the cases $n = 2$ and $n = 3$. If $n = 2$, using
  H\"{o}lder inequality with exponents $q > 2$ and $\frac{q}{q - 1}$ and the
  assumption that $| \widetilde{\mathcal{P}}_i \setminus \mathcal{P}_i |$ is
  of order $h^2$, we have
  \[ \| \overline{\nabla u}^i \|_{0, \widetilde{\mathcal{P}}_i \setminus
     \mathcal{P}_i} = \frac{1}{| \widetilde{\mathcal{P}}_i \setminus
     \mathcal{P}_i |^{\frac{1}{2}}} \left| \int_{\widetilde{\mathcal{P}}_i
     \setminus \mathcal{P}_i} \nabla u \right| \leqslant_{} \| \nabla u
     \|_{L^q (\widetilde{\mathcal{P}}_i \setminus \mathcal{P}_i)} |
     \widetilde{\mathcal{P}}_i \setminus \mathcal{P}_i |^{\frac{q - 1}{q} -
     \frac{1}{2}} \leqslant Ch^{1 - \varepsilon} \| \nabla u \|_{L^q
     (\widetilde{\mathcal{P}}_i \setminus \mathcal{P}_i)} \]
  with $\varepsilon = \frac{2}{q}$. Summing over all the patches, we have by
  the discrete H\"{o}lder inequality with exponents $\frac{q}{2}$ and $\frac{q /
  2}{q / 2 - 1}$ (recall that the number of patches $I$ is assumed uniformly
  bounded) and by the Sobolev embedding $H^1 (\Omega) \rightarrow L^q
  (\Omega)$
  \[ \left( \sum_i \| \overline{\nabla u}^i \|_{0, \widetilde{\mathcal{P}}_i
     \setminus \mathcal{P}_i}^2 \right)^{\frac{1}{2}} \leqslant Ch^{1 -
     \varepsilon} \left( \sum_i \| \nabla u \|^2_{L^q
     (\widetilde{\mathcal{P}}_i \setminus \mathcal{P}_i)}
     \right)^{\frac{1}{2}} \leqslant Ch^{1 - \varepsilon} \| \nabla u \|_{L^q
     (\Omega)} I^{\frac{1}{2} - \frac{1}{q}} \leqslant C_{\varepsilon} h^{1 -
     \varepsilon} \| \nabla u \|_{1, \Omega} \]
  with $C_{\varepsilon}$ depending both on $q$ (thus on $\varepsilon$) and on
  $\Omega$.
  
  Similarly, if $n = 3$, using H\"{o}lder inequality with exponents $6$ and
  $\frac{6}{5}$ and the assumption that $| \widetilde{\mathcal{P}}_i \setminus
  \mathcal{P}_i |$ is of order $h^3$, we have
  \[ \| \overline{\nabla u}^i \|_{0, \widetilde{\mathcal{P}}_i \setminus
     \mathcal{P}_i} \leqslant_{} \| \nabla u \|_{L^6
     (\widetilde{\mathcal{P}}_i \setminus \mathcal{P}_i)} |
     \widetilde{\mathcal{P}}_i \setminus \mathcal{P}_i |^{\frac{5}{6} -
     \frac{1}{2}} \leqslant Ch^{} \| \nabla u \|_{L^6
     (\widetilde{\mathcal{P}}_i \setminus \mathcal{P}_i)}. \]
  Summing over all the patches, we have by the discrete H\"{o}lder inequality
  with exponents 3 and $\frac{3}{2}$ and by the Sobolev embedding $H^1
  (\Omega) \rightarrow L^6 (\Omega)$
  \[ \left( \sum_i \| \overline{\nabla u}^i \|_{0, \widetilde{\mathcal{P}}_i
     \setminus \mathcal{P}_i}^2 \right)^{\frac{1}{2}} \leqslant Ch^{} \|
     \nabla u \|_{L^6 (\Omega)} I^{\frac{1}{3}} \leqslant Ch^{} \| \nabla u
     \|_{1, \Omega} \]
  with a constant depending on $\Omega$ and without introducing an additional
  parameter $\varepsilon$, \textit{i.e.} setting $\varepsilon = 0$. We conclude for
  both $n = 2$ and $3$
\[ \left( \sum_i |u|^2_{1, \widetilde{\mathcal{P}}_i \setminus \mathcal{P}_i}
   \right)^{\frac{1}{2}} \leqslant Ch \left( \sum_i |u|^2_{2,
   \widetilde{\mathcal{P}}_i \setminus \mathcal{P}_i} \right)^{\frac{1}{2}} +
   C_{\varepsilon} h^{1 - \varepsilon} \left( \sum_i \| \nabla u \|^2_{1,
   \widetilde{\mathcal{P}}_i \setminus \mathcal{P}_i} \right)^{\frac{1}{2}}
   \leqslant C_{\varepsilon} h^{1 - \varepsilon} \|u\|_{2, \Omega} .\]\end{proof}

\begin{proof}[Proof of Theorem \ref{theo:a priori gp}]
Let $e_h:=\widehat{\mathcal{I}}_hu-u_h$.
We remark that
\begin{align*}
\interleave e_h\interleave^2
&=a_h(e_h,e_h)\\ 
&= a_{\Omega^{nd}_h}(\widehat{\mathcal{I}}_hu-u_h,e_h)+\sum_i a_{\mathcal{P}_i}(\widehat{\mathcal{I}}_hu-\widehat{\mathcal{I}}_hu_h,\widehat{\mathcal{I}}_h  e_h)\\
&\hspace*{2cm}+\sum_i \frac{1}{h_{\mathcal{P}_i}^2}  ((\widehat{\mathcal{I}}_h-\mbox{Id})u_h, (\mbox{Id}-  \widehat{\mathcal{I}}_h) e_h)_{\mathcal{P}_i}.
\end{align*}
Lemma \ref{lemma:galer} leads to
\begin{equation}\label{est2}
\begin{aligned}
\interleave e_h\interleave^2
&= a_{\Omega^{nd}_h}(\widehat{\mathcal{I}}_hu-u,e_h)
+\sum_i a_{\mathcal{P}_i}(\widehat{\mathcal{I}}_hu,\widehat{\mathcal{I}}_h  e_h)-\sum_i a_{\mathcal{P}_i}(u, e_h)
\\
&=a_{\Omega^{nd}_h}(\widehat{\mathcal{I}}_hu-u,e_h)
+\sum_i a_{\mathcal{P}_i}(\widehat{\mathcal{I}}_hu-u,\widehat{\mathcal{I}}_he_h)+\sum_i a_{\mathcal{P}_i}(u,\widehat{\mathcal{I}}_he_h-e_h).
\end{aligned}
\end{equation}
 We now estimate each term in  the right-hand side.
 Using Proposition \ref{prop:global int} for the interpolation operator $\widehat{I}_h$, cf. Remark \ref{rem:global int}, it holds
\begin{equation}\label{est3}
\begin{aligned}
a_{\Omega^{nd}_h}(\widehat{\mathcal{I}}_hu-u,e_h)
&\leqslant |\widehat{\mathcal{I}}_hu-u|_{1,\Omega}|e_h|_{1,\Omega^{nd}_h}\\
&\leqslant Ch|u|_{2,\Omega}\interleave e_h\interleave
 \end{aligned}
\end{equation}
 and
\begin{equation}\label{est4}
\begin{aligned}
\sum_i  a_{\mathcal{P}_i}(\widehat{\mathcal{I}}_hu-u,\widehat{\mathcal{I}}_he_h)&\leqslant  \sum_i|\widehat{\mathcal{I}}_hu-u|_{1,\mathcal{P}_i}|\widehat{\mathcal{I}}_he_h|_{1,\mathcal{P}_i}\\
&\leqslant Ch|u|_{2,\Omega}\interleave e_h\interleave.
 \end{aligned}\end{equation} 
Concerning the third term, it holds
\begin{equation*}
\begin{aligned}
 \sum_ia_{\mathcal{P}_i}(u,\widehat{\mathcal{I}}_he_h-e_h)
&=\sum_i(\partial_nu,\widehat{\mathcal{I}}_he_h-e_h)_{0,\partial \mathcal{P}_i}-
\sum_i(\Delta u,\widehat{\mathcal{I}}_he_h-e_h)_{0,\mathcal{P}_i}\\
&=\sum_i(\Delta u,\widehat{\mathcal{I}}_he_h-e_h)_{0,\widetilde{\mathcal{P}}_i\backslash \mathcal{P}_i}
-\sum_i(\Delta u,\widehat{\mathcal{I}}_he_h-e_h)_{0,\mathcal{P}_i}
+\sum_ia_{\widetilde{\mathcal{P}}_i\backslash \mathcal{P}_i}(u,\widehat{\mathcal{I}}_he_h-e_h).
\end{aligned}\end{equation*}   
Since $h_{\widetilde{\mathcal{P}}_i}\leqslant Ch,$
we obtain the Poincar\'e type inequality
\begin{equation}\label{eq:poincare}
\|\widehat{\mathcal{I}}_he_h-e_h\|_{0,\widetilde{\mathcal{P}}_i\backslash \mathcal{P}_i}
\leqslant Ch\|\widehat{\mathcal{I}}_he_h-e_h\|_{1,\widetilde{\mathcal{P}}_i\backslash \mathcal{P}_i}.
\end{equation}  
By Cauchy-Schwarz inequality, inequality \eqref{eq:poincare}, 
Lemma \ref{lemma:patch} it holds 
\begin{equation}\label{est6}
\sum_i(\Delta u,\widehat{\mathcal{I}}_he_h-e_h)_{0,\widetilde{\mathcal{P}}_i\backslash \mathcal{P}_i}
\leqslant
\left(\sum_i|u|_{2,\widetilde{\mathcal{P}}_i\backslash \mathcal{P}_i}^2\right)^{1/2}\left(\sum_i\|\widehat{\mathcal{I}}_he_h-e_h\|_{0,\widetilde{\mathcal{P}}_i\backslash \mathcal{P}_i}^2\right)^{1/2}
\leqslant
Ch|u|_{2,\Omega}\interleave e_h\interleave.
\end{equation}  
Again, using Cauchy-Schwarz inequality, we obtain
\begin{equation}\label{est6bis}
\sum_i(\Delta u,\widehat{\mathcal{I}}_he_h-e_h)_{0,\mathcal{P}_i}
\leqslant
\left(\sum_i|u|_{2,\mathcal{P}_i}^2\right)^{1/2}\left(\sum_i\|\widehat{\mathcal{I}}_he_h-e_h\|_{0,\mathcal{P}_i}^2\right)^{1/2}
\leqslant
Ch|u|_{2,\Omega}\interleave e_h\interleave.\end{equation}  
By Lemmas \ref{lemma:patch} and \ref{NewLemma} 
\begin{equation}\label{est7}
\sum_i a_{\widetilde{\mathcal{P}}_i \setminus \mathcal{P}_i}  (u,
\widehat{\mathcal{I}}_h e_h - e_h) \leqslant 
\left( \sum_i |u|_{1,\widetilde{\mathcal{P}}_i \setminus \mathcal{P}_i}^2 \right)^{\frac{1}{2}}  \left( \sum_i |\widehat{\mathcal{I}}_h e_h - e_h |_{1, \widetilde{\mathcal{P}}_i \setminus \mathcal{P}_i}^2 \right)^{\frac{1}{2}}  
\leqslant C_{\varepsilon} h^{1 - \varepsilon}
     \| u \|_{2, \Omega} \interleave e_h \interleave .
\end{equation}
 Thus, Proposition \ref{prop:global int} for the interpolation operator $\widehat{I}_h$ and  \eqref{est2}, \eqref{est3}, 
 \eqref{est4}, \eqref{est6}, \eqref{est6bis}, \eqref{est7} lead to
  \begin{equation*}
  \interleave u-u_h\interleave
  \leqslant \interleave u-\widehat{\mathcal{I}}_hu\interleave+
  \interleave\widehat{\mathcal{I}}_hu-u_h\interleave\leqslant C_{\varepsilon}h^{1-\varepsilon} |u|_{2,\Omega}.
  \end{equation*}

  Consider now the solution $w \in V$ to
  \[ a (w, v) = (u - \widehat{\mathcal{I}}_h u_h, v)_{}, \quad \forall v \in
     V. \]
  Observe
  \[ a_{} (\widehat{\mathcal{I}}_h u_h, \widehat{\mathcal{I}}_h w) = a_h
     (u_h, \widehat{\mathcal{I}}_h w) - \sum_i a_{\widetilde{\mathcal{P}}_i
     \setminus \mathcal{P}_i} (u_h - \widehat{\mathcal{I}}_h u_h,
     \widehat{\mathcal{I}}_h w) = (f, \widehat{\mathcal{I}}_h w) - \sum_i
     a_{\widetilde{\mathcal{P}}_i \setminus \mathcal{P}_i} (u_h -
     \widehat{\mathcal{I}}_h u_h, \widehat{\mathcal{I}}_h w) \]
  so that
  \[ a_{} (u - \widehat{\mathcal{I}}_h u_h, \widehat{\mathcal{I}}_h w) =
     \sum_i a_{\widetilde{\mathcal{P}}_i \setminus \mathcal{P}_i} (u_h -
     \widehat{\mathcal{I}}_h u_h, \widehat{\mathcal{I}}_h w) = \sum_i
     a_{\widetilde{\mathcal{P}}_i \setminus \mathcal{P}_i} (e_h -
     \widehat{\mathcal{I}}_h e_h, \widehat{\mathcal{I}}_h w) \]
  with $e_h$=$u_h - \widehat{\mathcal{I}}_h u$. Thus,
 \begin{equation}\label{estimate12}
\begin{aligned} \| u - \widehat{\mathcal{I}}_h u_h \|_{0, \Omega}^2 &= a (u -
     \widehat{\mathcal{I}}_h u_h, w - \widehat{\mathcal{I}}_h w) + \sum_i
     a_{\widetilde{\mathcal{P}}_i \setminus \mathcal{P}_i} (e_h -
     \widehat{\mathcal{I}}_h e_h, \widehat{\mathcal{I}}_h w) \\
     & \leqslant Ch | u - \widehat{\mathcal{I}}_h u_h |_{1, \Omega} | w |_{2, \Omega} +
     \left( \sum_i | e_h - \widehat{\mathcal{I}}_h e_h |^2_{1,
     \widetilde{\mathcal{P}}_i \setminus \mathcal{P}_i} \right)^{\frac{1}{2}}
     \left( \sum_i | \widehat{\mathcal{I}}_h w |^2_{1,
     \widetilde{\mathcal{P}}_i \setminus \mathcal{P}_i} \right)^{\frac{1}{2}},
  \end{aligned} \end{equation}
  where we have used the interpolation estimate. Using Lemma
  \ref{lemma:galer} and (above), we obtain
  \[ \left( \sum_i | e_h - \widehat{\mathcal{I}}_h e_h |^2_{1,
     \widetilde{\mathcal{P}}_i \setminus \mathcal{P}_i} \right)^{\frac{1}{2}}
     \leqslant C \interleave e_h \interleave \leqslant C_{\varepsilon} h^{1 -
     \varepsilon} | u |_{2, \Omega} \]
  and
  \begin{align*}
  | u - \widehat{\mathcal{I}}_h u_h |_{1, \Omega} &\leqslant C \left( | u
     - u_h |^2_{1, \Omega^{{nd}}} + | u_h - \widehat{\mathcal{I}}_h u_h
     |^2_{1, \Omega^{{nd}}} + \sum_i | u - \widehat{\mathcal{I}}_h u_h
     |_{1, \mathcal{P}_i}^2 \right)^{\frac{1}{2}} \\
     &\leqslant C_{\varepsilon}
     h^{1 - \varepsilon} | u |_{2, \Omega} + \left( \sum_i | e_h -
     \widehat{\mathcal{I}}_h e_h |^2_{1, \widetilde{\mathcal{P}}_i \setminus
     \mathcal{P}_i} \right)^{\frac{1}{2}} \leqslant C_{\varepsilon} h^{1 -
     \varepsilon} | u |_{2, \Omega} .
  \end{align*} 
  We also have by regularity of elliptic problem in a convex polygon
  (polyhedron)
  \begin{equation*}
    |w|_{2, \Omega} \leqslant C \|u - \widehat{\mathcal{I}}_h
    u_h \|_{0, \Omega} 
  \end{equation*}
  and by Lemma \ref{NewLemma}
  \[ \left( \sum_i | \widehat{\mathcal{I}}_h w |^2_{1,
     \widetilde{\mathcal{P}}_i \setminus \mathcal{P}_i} \right)^{\frac{1}{2}}
     \leqslant | w - \widehat{\mathcal{I}}_h w |_{1, \Omega} + \left( \sum_i
     | w |^2_{1, \widetilde{\mathcal{P}}_i \setminus \mathcal{P}_i}
     \right)^{\frac{1}{2}} \leqslant Ch | w |_{2, \Omega} + C_{\varepsilon}
     h^{1 - \varepsilon} \| w \|_{2, \Omega} \leqslant C_{\varepsilon} h^{1 -
     \varepsilon} \|u - \widehat{\mathcal{I}}_h u_h \|_{0, \Omega} . \]
  Substituting into (\ref{estimate12}) gives
  \[ \| u - \widehat{\mathcal{I}}_h u_h \|_{0, \Omega}^2 \leqslant
     C_{\varepsilon} h^{2 - 2 \varepsilon} | u |_{2, \Omega} \|u -
     \widehat{\mathcal{I}}_h u_h \|_{0, \Omega} \]
  which in combination with the triangle inequality and the estimate for \ $\|
  u_h - \widehat{\mathcal{I}}_h u_h \|_{0, \mathcal{P}_i}$ contained in the
  estimate for $\interleave u - u_h \interleave$ gives the announced
  $L^2$-error estimate.

 \end{proof}

\subsection{Conditioning of the system matrix}

We are now going to prove that the conditioning number  of the finite element matrix  associated to 
the bilinear form $a_h$ of the alternative scheme does not deteriorate in the presence of degenerated cells: it is of order $1/h^2$ if the mesh is quasi-uniform in a sense specified below.

\begin{prpstn}[Conditioning]\label{prop:cond gp}
Suppose that Assumption \ref{cond: plat2} holds
and the union of mesh cells $\omega_x$ attached to each node $x$ of $\mathcal{T}_h$ satisfies
\begin{equation}\label{eq:cond2}
 c_1h^n\le |\omega_x|\le c_2h^n
\end{equation}
with some constants $c_1,c_2$.
Then, the conditioning number $\kappa(\boldsymbol{A})$ of the matrix $\boldsymbol{A}$ associated to 
the bilinear form $a_h$ in $V_h$ satisfies
\begin{equation*}
\kappa(\boldsymbol{A})\leqslant Ch^{-2}.
\end{equation*} 
\end{prpstn}

\begin{rmrk}
Condition  \eqref{eq:cond2} is satisfied for instance if the mesh is quasi-uniform in the sense $h_K\ge c_3h$ for all $K\in\mathcal{T}_h$, and each patch is constituted of a degenerated cell and a non-degenerated cell.
This is the situation considered in our numerical simulations given in Section \ref{sec:sim num}.
\end{rmrk}

Before proving Proposition \ref{prop:cond gp}, we first introduce some auxiliary results:


\begin{lmm}[Coercivity of $a_h$]\label{lemma:coer ah}
Under the assumptions of Proposition \ref{prop:cond gp}, it holds for all $v_h\in V_h$ 
$$a_h(v_h,v_h)\geqslant C \|v_h\|_{0,\Omega}^2.$$
\end{lmm}

\begin{proof}Let $v_h\in V_h$.
Observe, using triangle and Poincar\'e inequalities,
\begin{align*}
 \|v_h \|_{0, \Omega} &\leqslant \| \widehat{\mathcal{I}}_h v_h \|_{0, \Omega} 
    + \|v_h - \widehat{\mathcal{I}}_h v_h \|_{0, \Omega} \\
   &\leqslant C |   \widehat{\mathcal{I}}_h v_h |_{1, \Omega} 
    +\|v_h - \widehat{\mathcal{I}}_h v_h \|_{0, \Omega} \\
   &\leqslant C \left( | v_h
   |^2_{1, \Omega_h^{nd}} + \sum_i | \widehat{\mathcal{I}}_h v_h
   |_{1, \mathcal{P}_i}^2 + \sum_i |v_h - \widehat{\mathcal{I}}_h v_h |_{1,
   \widetilde{\mathcal{P}}_i \setminus \mathcal{P}_i}^2 + \sum_i \| v_h -
   \widehat{\mathcal{I}}_h v_h \|_{0, \widetilde{\mathcal{P}}_i}^2
   \right)^{\frac{1}{2}}  .
\end{align*}
Following the proof of Lemma \ref{lemma:patch}, we see easily
\[ \sum_i |v_h - \widehat{\mathcal{I}}_h v_h |_{1, \widetilde{\mathcal{P}}_i
   \setminus \mathcal{P}_i}^2 + \sum_i \| v_h - \widehat{\mathcal{I}}_h v_h
   \|_{0, \widetilde{\mathcal{P}}_i \setminus \mathcal{P}_i}^2 \leqslant C
   \interleave v_h \interleave^2. \]
This implies $\|v_h \|_{0, \Omega}^{} \leqslant C \interleave v_h \interleave$
which is equivalent to the desired result.

\end{proof}

\begin{lmm}[Continuity of $a_h$]\label{lemma:cont ah}
Under the assumptions of Proposition \ref{prop:cond gp}, it holds for all $u_h,v_h\in V_h$
$$a_h(u_h,v_h)\leqslant \frac{C}{h^2} \|u_h\|_{0,\Omega}\|v_h\|_{0,\Omega}.$$
\end{lmm}

\begin{proof}Let $u_h,v_h\in V_h$.
Since the cells of $\Omega^{nd}_h$ and the patches $\mathcal{P}_i$ are regular, we obtain using the inverse inequality
\[ a_{\Omega^{nd}_h} (u_h, v_h) \leqslant \frac{C}{h^2} \| u_h \|_{0,\Omega^{nd}_h} \| v_h \|_{0,\Omega^{nd}_h} \]
  and 
   \[ a_{\mathcal{P}_i}
   (\widehat{\mathcal{I}}_h u_h, \widehat{\mathcal{I}}_h v_h) 
   \leqslant \frac{C}{h^2} \| \widehat{\mathcal{I}}_hu_h \|_{0,\mathcal{P}_i} \| \widehat{\mathcal{I}}_hv_h \|_{0,\mathcal{P}_i}.\]
Using the equivalence 
of the norm in finite dimensional spaces and the fact that 
$\mathcal{P}_i$ and $K_i^{nd}$ are regular,
for all $w_h\in V_h$, it holds
$$
\|\widehat{\mathcal{I}}_h w_h\|_{0,\mathcal{P}_i}\leqslant C \|w_h\|_{0,K_i^{nd}}
.$$
We deduce that
  \[
   \|(\mbox{Id}- \widehat{\mathcal{I}}_h) w_h\|_{0,\mathcal{P}_i} 
   +\| \widehat{\mathcal{I}}_hw_h \|_{0,\mathcal{P}_i}
\leqslant  C(\|  w_h \|_{0,\mathcal{P}_i}
+\| w_h \|_{0,K_i^{nd}})\leqslant C  \|  w_h \|_{0,\mathcal{P}_i}
   \] 
   which leads to the conclusion.
  \end{proof}
  

\begin{proof}[Proof of Proposition \ref{prop:cond gp}]
We first remark using  \eqref{eq:cond2}, that
there exists $C_1,C_2>0$ such that for all $w_h\in V_h$ 
and $\boldsymbol{w}$
its associated vector in $\mathbb{R}^{ N}$
\begin{equation}\label{eq:vh v}
C_1h^{n/2}|\boldsymbol{w}|_2\leqslant \|w_h\|_0\leqslant C_2h^{n/2}|\boldsymbol{w}|_2.
\end{equation}
Indeed, denoting by $\mathcal{N}_h$ the set of nodes of $\mathcal{T}_h$, by $\mathcal{N}_h(K)$ the set of nodes of a simplex $K\in\mathcal{T}_h$, and using $\sim$ to denote the equivalence with universal constant, as in (\ref{eq:vh v}), we can conclude
$$
\|w_h\|_0^2 \sim \sum_{K\in\mathcal{T}_h}|K|\sum_{x\in\mathcal{N}_h(K)}|w_h(x)|^2
=\sum_{x\in\mathcal{N}_h}|w_h(x)|^2|\omega_x| \sim h^{n}|\boldsymbol{w}|_2^2.
$$
In what follows,  $\boldsymbol{v}\in \mathbb{R}^{ N}$ denotes the vector associated to $v_h\in V_h$. Inequality \eqref{eq:vh v} with Lemma \ref{lemma:cont ah} imply
$$
\|\boldsymbol{A}\|_2=\sup_{\boldsymbol{v}\in\mathbb{R}^{ N}}\dfrac{(\boldsymbol{A}\boldsymbol{v},\boldsymbol{v})}{|\boldsymbol{v}|_2^2}=\sup_{\boldsymbol{v}\in\mathbb{R}^{ N}}\dfrac{a(v_h,v_h)}{|\boldsymbol{v}|_2^2}
\leqslant Ch^n\sup_{v_h\in V_h}\dfrac{a(v_h,v_h)}{ \|v_h\|_0^2}
\leqslant Ch^{n-2}.
$$
Similarly, \eqref{eq:vh v} with Lemma \ref{lemma:coer ah} imply
$$\|\boldsymbol{A}^{-1}\|_2
=\sup_{\boldsymbol{v}\in\mathbb{R}^{ N}}\dfrac{|\boldsymbol{v}|_2^2}{(\boldsymbol{A}\boldsymbol{v},\boldsymbol{v})} = \sup_{\boldsymbol{v}\in\mathbb{R}^{ N}}\dfrac{|\boldsymbol{v}|_2^2}{a(v_h,v_h)}
\leqslant Ch^{-n}\sup_{v_h\in V_h}\dfrac{ \|v_h\|_0^2}{a(v_h,v_h)}
\leqslant Ch^{-n}.
$$
These estimates lead to the desired result.
\end{proof}

\subsection{An equivalent, easily implementable  variational formulation with interior penalty}\label{sec:ghost penalty}

Since implementing the interpolation operator $\widehat{\mathcal{I}}_h$ is not necessary trivial, we rewrite in this section the bilinear form $a_h$ given in \eqref{eq:bil GP} in an equivalent form, which introduces the jumps of the gradients over the interior facets. The resulting method is similar to the ghost penalty from \cite{burmancras}. 

\begin{lmm}
Under Assumption \ref{cond: plat2}, suppose moreover that each patch $\mathcal{P}_i$ 
is composed of a non-degenerated cell $K_i^{nd}$
and a degenerated cell $K^{deg}_i$. 
Denote by $F_i$ the facet between $K_i^{nd}$ and $K^{deg}_i$, as
 illustrated in Fig.  \ref{fig:patch}.
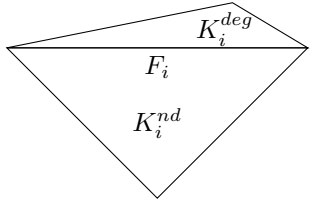
\begin{figure}[H]\begin{center}
    \begin{tikzpicture}[scale=1]
\draw (0,0) -- (4,0) -- (2,-2) -- cycle;
\draw (0,0) -- (4,0) -- (3,0.6) -- cycle;
\path (2,-1) node {$K^{nd}_i$};
\path (2.9,0.25) node {$K^{deg}_i$};
\path (2,-0.25) node {$F_i$};
\end{tikzpicture}
\caption{Example of patch $\mathcal{P}_i=K_i^{nd}\cup K^{deg}_i$.}
\label{fig:patch}
\end{center}\end{figure}
Then, for all $u_h,v_h\in V_h$, it holds 
\begin{equation}\label{eq:equi form}
 a_h (u_h, v_h) = a_{\Omega^{nd}_h} (u_h, v_h) + \sum_i \frac{|\mathcal{P}_i|}{|K_i^{nd}|}a_{K_i^{nd}}   ( u_h, v_h) 
   + \kappa_n\sum_i \frac{|K^{deg}_i|^3}{h_{\mathcal{P}_i}^2|F_i|^2}[\nabla u_h]_{F_i}\cdot[\nabla v_h]_{F_i}
   \end{equation}
with $\kappa_n:=\frac{2n^2}{(n+1)(n+2)}$.   
\end{lmm}

\begin{proof}
Let us assume, without loss of generality, that the coordinate axes are chosen so that the $y$ axis is orthogonal to $F_i$, as in Fig. \ref{fig:bad cell}. We also denote by $h_i$ the height of the simplex $K_i^{deg}$ drawn to the base $F_i$.
\begin{figure}[H]\begin{center}
    \begin{tikzpicture}[scale=1]
\draw (0,0) -- (4,0) -- (3,0.6) -- cycle;
\draw[<->,dashed] (4.2,0) -- (4.2,0.6);
\draw[->] (0,0) -- (5,0);
\draw[->] (0,0) -- (0,1);
\path (2.9,0.25) node {$K^{deg}_i$};
\path (0,1.2) node {$y$};
\path (5.2,0) node {$x$};
\path (2,-0.25) node {$F_i$};
\path (4.5,0.3) node {$h_i$};
\end{tikzpicture}
\caption{Degenerated cell $K^{deg}_i$. 
}
\label{fig:bad cell}
\end{center}\end{figure}
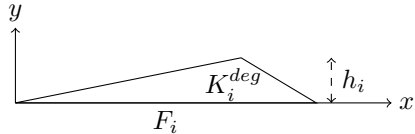
We first remark that, for all $u_h\in V_h$,
\begin{equation*}
(\mbox{Id}-\widehat{\mathcal{I}}_h)u_h=\left\{\begin{array}{ll}
\left[\nabla u_h\right]_{F_i}y&\mbox{ on } K^{deg}_i,\\
0&\mbox{ on } K_i^{nd}.
\end{array}\right.
\end{equation*}
Hence, we deduce that
\[
((\mbox{Id}- \widehat{\mathcal{I}}_h) u_h, (\mbox{Id}-\widehat{\mathcal{I}}_h) v_h)_{\mathcal{P}_i}
=\left[\nabla u_h\right]_{F_i} \cdot \left[\nabla v_h\right]_{F_i}\int_{K^{deg}_i}y^2.
\]
Moreover 
\begin{equation*}
\int_{K^{deg}_i}y^2
=\displaystyle\int_{0}^{h_i} y^2 |F_i| \left(1-\frac{y}{h_i}\right)^{n-1}\,dy
=|F_i|h_i^3\frac{2}{n(n+1)(n+2)}
=\frac{|K^{deg}_i|^3}{|F_i|^2}\frac{2n^2}{(n+1)(n+2)},
\end{equation*}
since $|K^{deg}_i|=\frac{1}{n}|F_i|h_i$. This leads to the conclusion.

\end{proof}

\section{Numerical simulations}\label{sec:sim num}

In this section, we will illustrate with some numerical examples the sharpness of the \textit{a priori} estimates of Theorem \ref{th: a priori} 
and  the efficiency of the method proposed in Section \ref{sec:ghost penalty} to ensure the good conditioning of the matrix.
The simulations of this section have been implemented using the finite element library FEniCS \cite{LoggMardal2012}.


We consider problem \eqref{eq:poisson} on the domain $\Omega:=(0,1)\times(0,1)$ with the right hand side 
$f(x,y)=2\pi^2\sin(\pi x)\sin(\pi y)$ so that the exact solution is given by $u(x,y)=\sin(\pi x)\sin(\pi y)$
for $(x,y)\in \Omega$. To construct the meshes $\mathcal{T}_h$ in all our numerical experiments presented below,  we start from a uniform Cartesian mesh of step $h$ and degenerate certain cells so that for each degenerated cell $K^{deg}$, 
$h_{K^{deg}}=h$, $\rho_{K^{deg}}\sim h^2$
(more precisely, the distance between the longer side and the opposite node will be equal to $h^2$).
In doing so, we take care that each degenerated cell be included in a patch of surrounding cells, and the patches corresponding to distinct degenerated cells do not intersect each other, cf. Fig. \ref{simu:mesh}. Assumptions \ref{cond: plat} and \ref{cond: plat2} are thus satisfied.


We report in Fig. \ref{fig:simu no ghost} the numerical results obtained on a series of meshes with decreasing $h$, taking 10 degenerated cells (as described above) for every $h$. We use here the standard scheme \eqref{discret prob} to produce the approximated solution $u_h$. The $L^2$ and $H^1$ absolute errors between $u$ and $u_h$ are given on the left in Fig.~\ref{fig:simu no ghost}. The optimal convergence rates are indeed observed, as predicted by Theorem \ref{th: a priori}. However, the conditioning number of the associated finite element matrix is much bigger than $1/h^2$, which would be expected on a quasi-uniform mesh with step $h$. This is illustrated by Fig. \ref{fig:simu no ghost}, right. The estimate on the conditioning number from Proposition \ref{prop:poor} is recovered, \textit{i.e.} $\kappa(A)\sim 1/(h\varepsilon) \sim 1/h^3$, since  $\varepsilon=\rho_{K^{deg}}\sim h^2$.

We now turn to the alternative scheme \eqref{disc prob GP}. We have implemented it using the reformulation \eqref{eq:equi form}.
The results are reported in Fig.~\ref{fig:simu ghost} using the same meshes containing 10 degenerated cells as above. The errors are reported on the left. We recall that Theorem \ref{theo:a priori gp} predicts the optimal convergence in the $H^1$ norm only if the approximate solution $u_h$ is post-processed on the degenerated cells, by replacing the actual polynomial giving $u_h$ on such a cell by the extension $\Pi_hu_h$ of $u_h$ from the attached regular cell, cf. the definition of $|u-\Pi_hu_h|_{1,\Omega}$ in \eqref{eq:a priori GP}. Numerical experiments confirm the optimal $H^1$ convergence of the post-processed solution and also the necessity of such a post-processing. Indeed, the error with respect to the non-processed approximate solution $|u-u_h|_{1,\Omega}$ is not of optimal order $h$. It is also much bigger than $|u-\Pi_hu_h|_{1,\Omega}$. We also note that the optimal $L^2$ convergence is recovered without any post-processing, as predicted by Theorem \ref{theo:a priori gp}. We recall that the introduction of the alternative scheme \eqref{disc prob GP} was motivated by the desire to obtain less ill-conditioned matrices. The results in Fig. \ref{fig:simu ghost} (right) confirm that conditioning number for this scheme is indeed no longer affected by the presence of degenerated cells, in accordance with Proposition \ref{prop:cond gp}.

We recall that the theory of Section \ref{sec:alter} concerning the alternative scheme \eqref{disc prob GP} is developed under Assumption \ref{cond: plat2} supposing, in particular, that the number of degenerate cells is uniformly bounded. In the numerical experiments reported in Figs. \ref{simu:mesh full} and \ref{fig:simu full}, we wish to check if such an assumption is indeed necessary. We consider to this end a sequence of meshes constructed as above, but containing an increasing number of degenerate cells, cf. Fig.~\ref{simu:mesh full}. We consider namely the densest packing of the degenerated cells allowed by Assumption \ref{cond: plat} (the non-intersection of the surrounding patches), which gives approximately 5.5\% of degenerated cells. Otherwise, the procedure for degenerating the cells is as above, in particular, $\rho_{K^{deg}_i}\approx h^2$. 
The results are presented in Fig.~\ref{fig:simu full} both for the standard scheme on the left, and the alternative scheme \eqref{disc prob GP} on the right. We first remark that the standard scheme remains optimally convergence in $L^2$ and $H^1$, in accordance with Theorem \ref{th: a priori}. On the contrary, the alternative scheme \eqref{disc prob GP} does not converge. This observation highlights the sharpness of the results given in Theorem \ref{theo:a priori gp}.


%
\begin{figure}
\begin{center}
\includegraphics[width=0.45\textwidth]{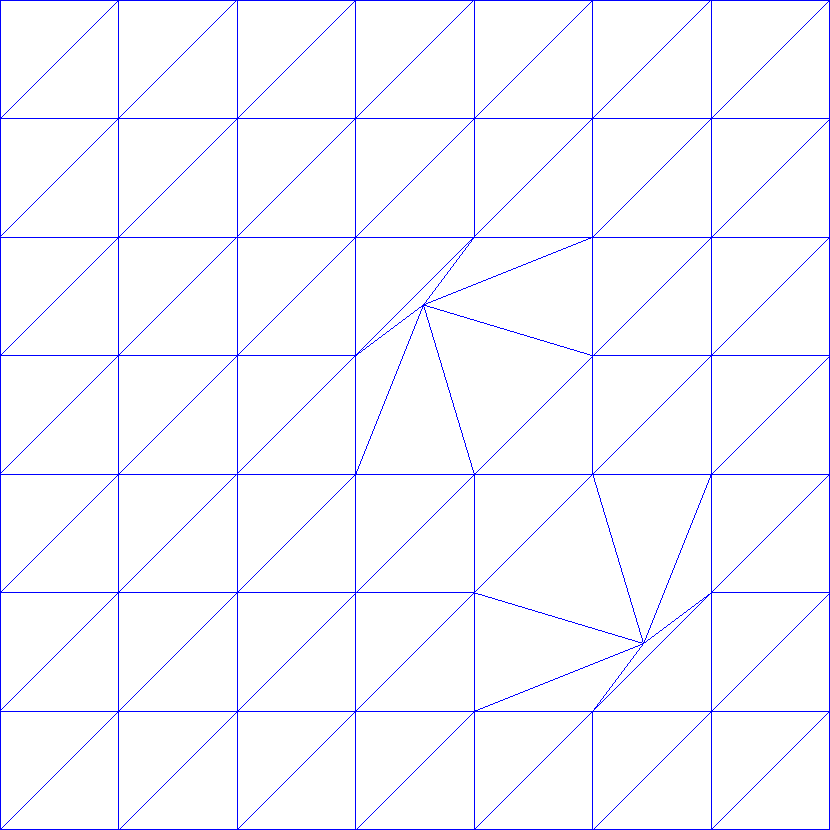}
\hspace*{1mm}
\includegraphics[width=0.45\textwidth]{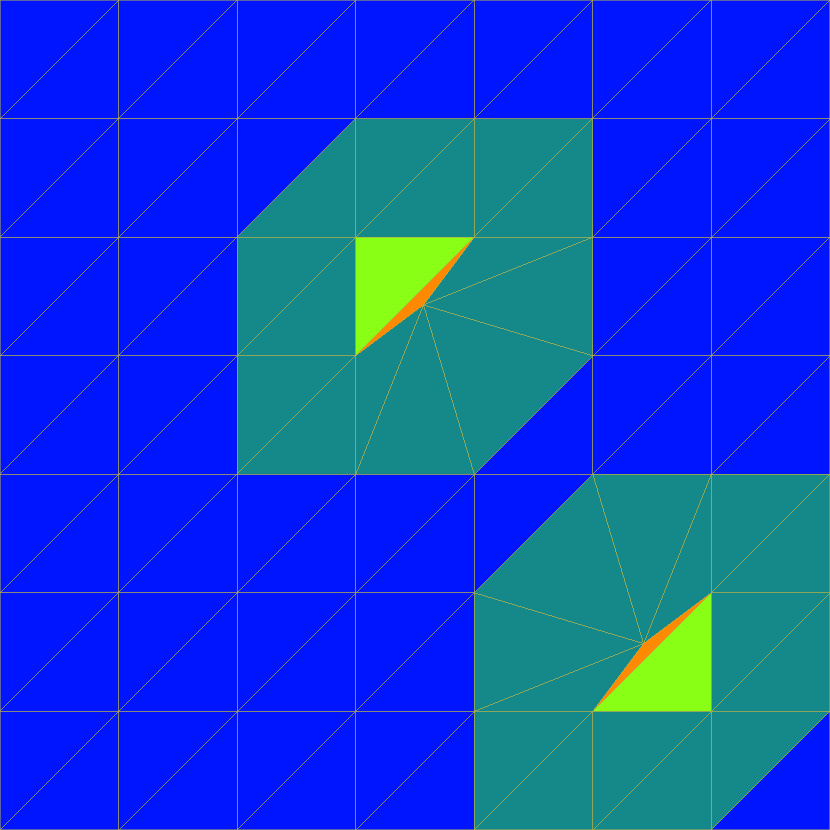}
\end{center}
\caption{Example of a uniform mesh with 2 arbitrarily chosen degenerated cells $K^{deg}_i$, $\rho_{K^{deg}_i}\sim h^2$ (left). On the right, the degenerated cells are painted in red, the adjacent regular cells in light green, and the surrounding patches in dark green. 
}
\label{simu:mesh}
\end{figure}

\begin{figure}
\begin{center}
\begin{tikzpicture}[thick,scale=0.85, every node/.style={scale=1.0}]
\begin{loglogaxis}[xlabel=$h$,xmin=1e-3,xmax=0.1
,legend pos=south east
,legend style={ font=\large}
, legend columns=1]
 \addplot[color=blue,mark=triangle*] coordinates { 
(0.0625,0.00591807395578)
(0.037037037037,0.00200486493141)
(0.0227272727273,0.000733311783114)
(0.013698630137,0.000260727150577)
(0.00826446280992,9.48617059382e-05)
(0.005,3.46668703305e-05)
(0.00302114803625,1.26445070934e-05)
(0.00183486238532,4.66319402008e-06)
 };
 \addplot[color=red,mark=triangle*] coordinates { 
(0.0625,0.227631436842)
(0.037037037037,0.132527028446)
(0.0227272727273,0.0801291009057)
(0.013698630137,0.0478752570248)
(0.00826446280992,0.0288737866471)
(0.005,0.0174566741603)
(0.00302114803625,0.0105434842572)
(0.00183486238532,0.00640292566017)
 };
\logLogSlopeTriangle{0.5}{0.2}{0.65}{1}{red};
\logLogSlopeTriangle{0.5}{0.2}{0.1}{2}{blue};

 \legend{$\|u-u_h\|_{0,\Omega}$,$|u-u_h|_{1,\Omega}$}
\end{loglogaxis}
\end{tikzpicture}
\begin{tikzpicture}[thick,scale=0.85, every node/.style={scale=1.0}]
\begin{loglogaxis}[xlabel=$h$,xmin=0.003,xmax=0.1,ymin=0.1,ymax=50
,legend pos=south east
,legend style={ font=\large}
, legend columns=1]
 \addplot[color=blue,mark=triangle*] coordinates { 
(0.0625,1.52852159484)
(0.037037037037,2.35115070965)
(0.0227272727273,3.63738037747)
(0.013698630137,5.84036637178)
(0.00826446280992,9.48782094788)
(0.005,15.4909495445)
};
 \logLogSlopeTriangleinv{0.5}{0.2}{0.5}{1}{blue};
 \legend{Conditioning $\times \frac{1}{h^2}$}
\end{loglogaxis}
\end{tikzpicture}
\caption{Errors (left) and conditioning (right) for the standard finite element scheme \eqref{discret prob} on a sequence of meshes containing 10 degenerated cells  with $\rho_{K^{deg}_i}\sim h^2$.}\label{fig:simu no ghost}
\end{center}\end{figure}
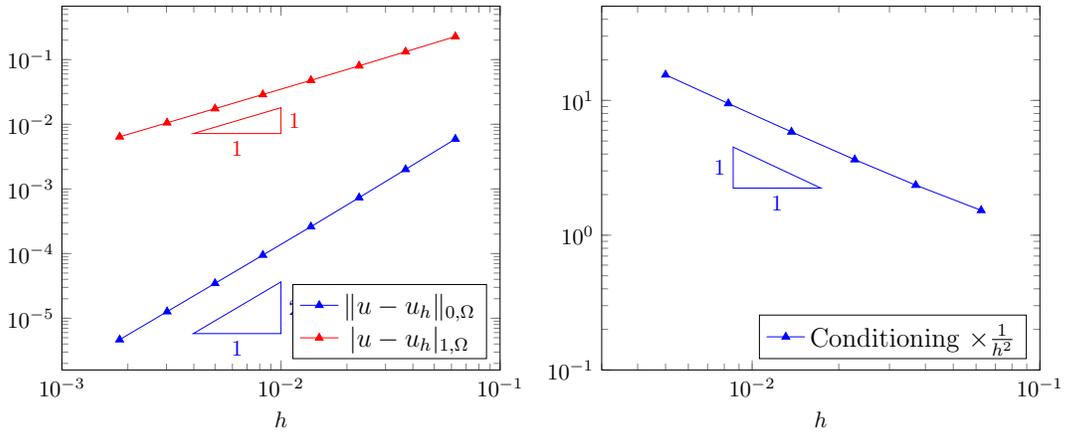


\begin{figure}
\begin{center}
\begin{tikzpicture}[thick,scale=0.85, every node/.style={scale=1.0}]
\begin{loglogaxis}[xlabel=$h$,xmin=0.001,xmax=0.1
,legend pos=south east
,legend style={ font=\large}
, legend columns=1]
 \addplot[color=blue,mark=triangle*] coordinates { 
(0.0625,0.00528574421618)
(0.037037037037,0.00195392227992)
(0.0227272727273,0.000715020945069)
(0.013698630137,0.000270040998205)
(0.00826446280992,9.6267651369e-05)
(0.005,3.61526213146e-05)
(0.00302114803625,1.3163167625e-05)
(0.00183486238532,4.88222978312e-06)
};
 \addplot[color=red,mark=triangle*] coordinates { 
(0.0625,0.425267054005)
(0.037037037037,0.284322649985)
(0.0227272727273,0.187926936202)
(0.013698630137,0.159970710161)
(0.00826446280992,0.108642894084)
(0.005,0.0866477328645)
(0.00302114803625,0.0558798561826)
(0.00183486238532,0.0437111656367)

 };
  \addplot[color=green,mark=triangle*] coordinates { 
  (0.0625,0.23611888184)
(0.037037037037,0.133223348021)
(0.0227272727273,0.0808597827443)
(0.013698630137,0.0482812867443)
(0.00826446280992,0.0291132717014)
(0.005,0.0176208327571)
(0.00302114803625,0.0106288590634)
(0.00183486238532,0.00645337326594)
 };
\logLogSlopeTriangle{0.5}{0.2}{0.6}{1}{green};

\logLogSlopeTriangle{0.5}{0.2}{0.1}{2}{blue};
 \legend{$\|u-u_h\|_{0,\Omega}$,$|u-u_h|_{1,\Omega}$,$|u-\Pi_hu_h|_{1,\Omega}$}
\end{loglogaxis}
\end{tikzpicture}
\begin{tikzpicture}[thick,scale=0.85, every node/.style={scale=1.0}]
\begin{loglogaxis}[xlabel=$h$,xmin=0.003,xmax=0.1,ymin=0.1,ymax=10
,legend pos=south east
,legend style={ font=\large}
, legend columns=1]
 \addplot[color=blue,mark=triangle*] coordinates { 
(0.0625,0.425731450503)
(0.037037037037,0.427411668339)
(0.0227272727273,0.41755323515)
(0.013698630137,0.417314886056)
(0.00826446280992,0.421930428057)
(0.005,0.418247903402)
 };
 \legend{Conditioning $\times \frac{1}{h^2}$}
\end{loglogaxis}
\end{tikzpicture}

\caption{Errors (left) and conditioning (right) for the alternative finite element scheme \eqref{disc prob GP} on a sequence of meshes containing 10 degenerated cells  with $\rho_{K^{deg}_i}\sim h^2$. The $H^1$ norm is calculated both using the approximate solution $u_h$ directly and extending it to the degenerated cells from the adjacent regular cells, as in \eqref{eq:a priori GP}.}\label{fig:simu ghost}
\end{center}\end{figure}
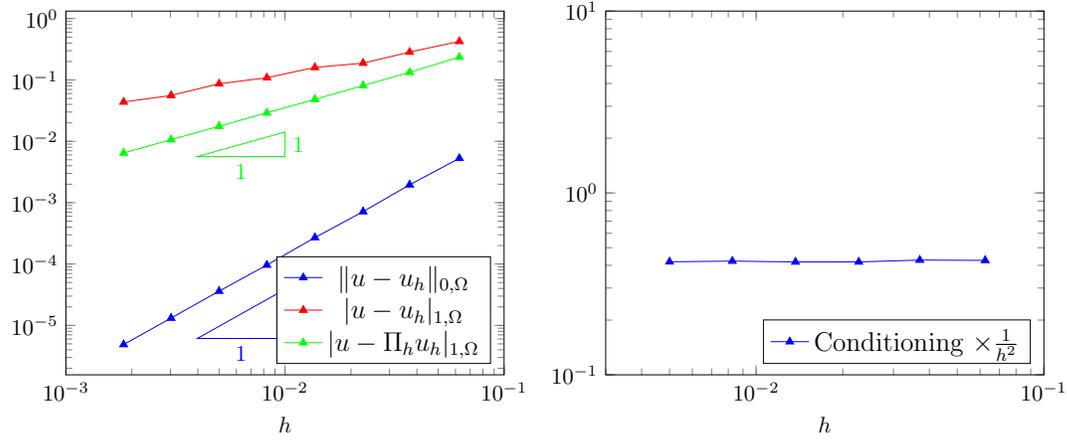

\begin{figure}
\begin{center}
\includegraphics[width=0.4\textwidth]{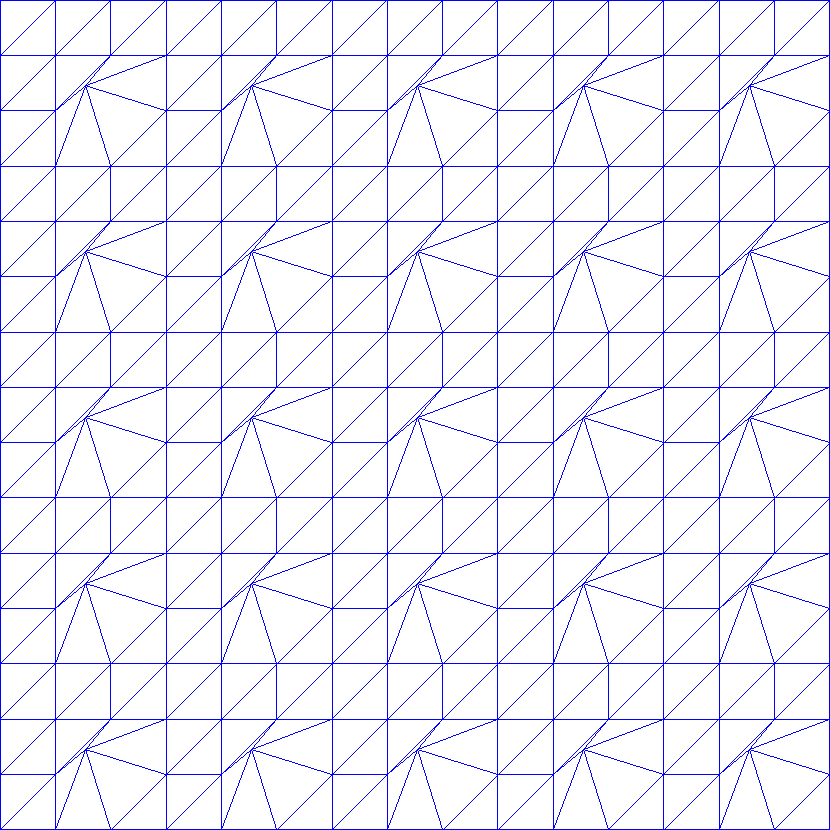}
\hspace*{1mm}
\includegraphics[width=0.4\textwidth]{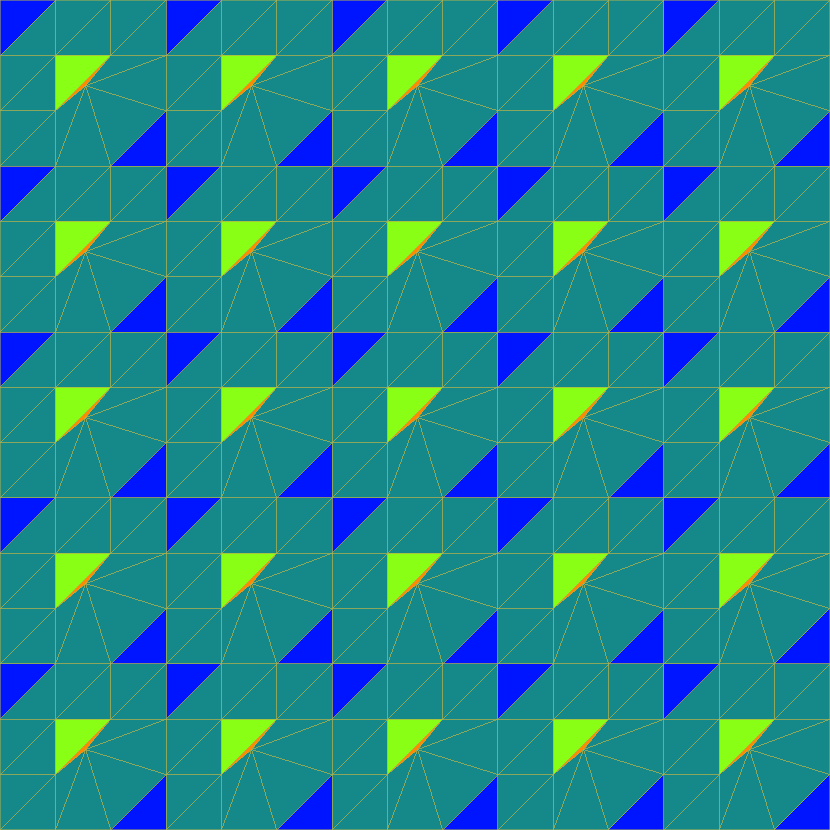}
\end{center}
\caption{Example of a mesh with densely packed degenerated cells ($\approx$ 5.5\% of degenerated cells).
Left: the mesh. Right: the disjoint patches surrounding the degenerated cells.}
\label{simu:mesh full}
\end{figure}

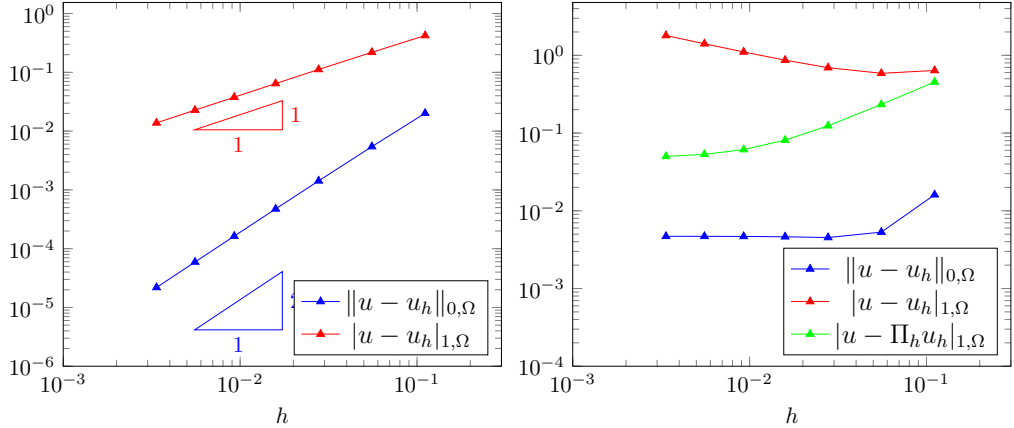
\begin{figure}
\begin{center}
\begin{tikzpicture}[thick,scale=0.85, every node/.style={scale=1.0}]
\begin{loglogaxis}[xlabel=$h$,xmin=1e-3,xmax=0.3 ,ymin=1e-6
,legend pos=south east
,legend style={ font=\large}
, legend columns=1]
 \addplot[color=blue,mark=triangle*] coordinates { 
(0.111111111111,0.0201517258154)
(0.0555555555556,0.00544753881014)
(0.0277777777778,0.001423694032)
(0.015873015873,0.000474569374132)
(0.00925925925926,0.000163431263844)
(0.00555555555556,5.92419132919e-05)
(0.003367003367,2.1850291539e-05)
 };
 \addplot[color=red,mark=triangle*] coordinates { 
(0.111111111111,0.422930378543)
(0.0555555555556,0.219230849866)
(0.0277777777778,0.111796737639)
(0.015873015873,0.064458939686)
(0.00925925925926,0.0377950531277)
(0.00555555555556,0.0227438110237)
(0.003367003367,0.0138083905613)
 };

\logLogSlopeTriangle{0.5}{0.2}{0.65}{1}{red};

\logLogSlopeTriangle{0.5}{0.2}{0.1}{2}{blue};

 \legend{$\|u-u_h\|_{0,\Omega}$,$|u-u_h|_{1,\Omega}$}
\end{loglogaxis}
\end{tikzpicture}
\begin{tikzpicture}[thick,scale=0.85, every node/.style={scale=1.0}]
\begin{loglogaxis}[xlabel=$h$,xmin=0.001,xmax=0.3,ymin=1e-4
,legend pos=south east
,legend style={ font=\large}
, legend columns=1]
 \addplot[color=blue,mark=triangle*] coordinates { 
(0.111111111111,0.0160770073613)
(0.0555555555556,0.00531152800099)
(0.0277777777778,0.00452794821245)
(0.015873015873,0.00463008253513)
(0.00925925925926,0.00467739198935)
(0.00555555555556,0.0046911963673)
(0.003367003367,0.00469410392502)
};
 \addplot[color=red,mark=triangle*] coordinates { 
(0.111111111111,0.638447877887)
(0.0555555555556,0.587648638013)
(0.0277777777778,0.694069940631)
(0.015873015873,0.866852944079)
(0.00925925925926,1.10595484942)
(0.00555555555556,1.4098386501)
(0.003367003367,1.79857130098)
 };
  \addplot[color=green,mark=triangle*] coordinates { 
  (0.111111111111,0.4541)
(0.0555555555556,0.2334)
(0.0277777777778,0.1241)
(0.015873015873,0.0811)
(0.00925925925926,0.0613)
(0.00555555555556,0.0532340817064)
(0.003367003367,0.0500879512663)
 };

 \legend{$\|u-u_h\|_{0,\Omega}$,$|u-u_h|_{1,\Omega}$,$|u-\Pi_hu_h|_{1,\Omega}$}
\end{loglogaxis}
\end{tikzpicture}
\caption{Errors on the meshes containing $\approx$ 5.5\% of degenerated cells.
Left: the standard scheme \eqref{discret prob}. Right: alternative scheme \eqref{disc prob GP}. }\label{fig:simu full}
\end{center}\end{figure}

\section*{Acknowledgements}
The authors are thankful to Marek Bucki (TexiSense) and Franz Chouly (Universit\'e de Bourgogne Franche-Comt\'e) for inspiring discussions which were at the origin of this project and for constant support during its realization. 

\bibliographystyle{plain}
\bibliography{biblio}
\end{document}